\documentclass[10pt]{article}      % Comments after  % are ignored
\usepackage{amsmath,amssymb,amsfonts,amsthm,amscd} % Typical maths resource packages
\usepackage{graphics}                 % Packages to allow inclusion of graphics
\usepackage{color}                    % For creating coloured text and background
\usepackage{ mathrsfs }
\usepackage{fullpage}
\usepackage{bbold}
\usepackage{enumitem}
\usepackage{url}         %to include url's in citations
\usepackage{colonequals} %to use the := symbol
\usepackage{fullpage}
\usepackage[hidelinks]{hyperref}

\usepackage{booktabs} % For prettier tables

\usepackage{tikz}
\usetikzlibrary{matrix}
\allowdisplaybreaks % For breaking equations

\theoremstyle{plain}
\newtheorem{theorem}{Theorem}[section]
\newtheorem{thm}[theorem]{Theorem}

\newtheorem{lem}[theorem]{Lemma}
\newtheorem{prop}[theorem]{Proposition}

\theoremstyle{definition}
\newtheorem{assume}[theorem]{Assumption}

\newtheorem{rem}[theorem]{Remark}
\usepackage{etoolbox} % Required for the next command for adding QED after remarks
\AtEndEnvironment{rem}{\qed}
\theoremstyle{remark}

\numberwithin{equation}{section}
\newcommand{\R}{\mathbb{R}}

\newcommand\dps{\displaystyle}

\newcommand\bra[1]{\left({#1}\right)}
\newcommand\pra[1]{\left[{#1}\right]}

\newcommand\abs[1]{\left\lvert#1\right\rvert}

\DeclareMathOperator{\Id}{Id}

\DeclareMathOperator{\PI}{PI}

\DeclareMathOperator{\LSI}{LSI}

\DeclareMathOperator{\sym}{sym}

\def\div{\mathop{\mathrm{div}}\nolimits}

\def\vep{\varepsilon}

\def\RelEnt{\mathscr H}

\def\E{\mathbb{E}}

\usepackage{tocloft}
\date{\today}
\title{Effective dynamics for non-reversible stochastic differential equations: a quantitative study}
\author{Fr{\'e}d{\'e}ric Legoll, Tony Leli{\`e}vre,  Upanshu Sharma}

\begin{document}
\maketitle

\begin{abstract}
Coarse-graining is central to reducing dimensionality in molecular dynamics, and is typically characterized by a mapping which projects the full state of the system to a smaller class of variables. While extensive literature has been devoted to coarse-graining starting from reversible systems, not much is known in the non-reversible setting. In this article, starting with a non-reversible dynamics, we introduce and study an \emph{effective} dynamics which approximates the (non-closed) \emph{projected} dynamics. Under fairly weak conditions on the system, we prove error bounds on the trajectorial error between the projected and the effective dynamics. In addition to extending existing results to the non-reversible setting, our error estimates also indicate that the notion of \emph{mean force} motivated by this effective dynamics is a good one.  
\end{abstract}

\section{Introduction}
Coarse-graining is an umbrella term used for techniques which approximate large and complex systems by simpler and lower-dimensional ones. Such procedures are fundamental tools in computational statistical mechanics where an all-atom molecular simulation of a complex system is often unable to access information about relevant temporal and/or spatial scales. These techniques are also important from a  modelling perspective as quantities of interest are often described by a smaller class of variables. 
 
Typically, coarse-graining requires scale separation, i.e. the presence of fast and slow scales. In this setting, as the ratio of fast to slow increases, the fast variables remain at equilibrium with respect to the slow ones,  and therefore the right choice for the coarse-grained variables are the slow ones. Such a situation has been dealt with via various techniques: the Mori-Zwanzig projection formalism~\cite{Grabert06,GivonKupfermanStuart04}, Markovian approximation to Mori-Zwanzig projections~\cite{HijonEspanolEijndenDelgado10}, averaging techniques~\cite{Hartmann07,PavliotisStuart08} and variational techniques~\cite{DuongLamaczPeletierSharma17} to name a few. Recently, in~\cite{BerezhkovskiiSzabo13,LuEijnden14}, coarse-graining techniques based on committor functions have been developed for situations where scale separation is absent but for very specific coarse-graining maps. 
 
As indicated by the literature above and the extensive references therein, the question of qualitative and quantitative coarse-graining has received wide attention over the years. However the question of deriving explicit error estimates, especially in the absence of explicit scale separation, is a challenging one. Starting from a reversible dynamics and a given coarse-graining map, in~\cite{LL10,LegollLelievre12} the authors propose an \emph{effective dynamics} which approximates the coarse-grained variables and prove quantitative error estimates comparing the time marginals of the coarse-grained variables and the effective dynamics. Recently, in~\cite{DLPSS17}, the authors generalize the ideas in~\cite{LL10,LegollLelievre12} to a wider setting, including the error estimates on time marginals starting with the (underdamped) Langevin equation.  Although these estimates demonstrate the accuracy of the effective dynamics, they do not provide any correlation-in-time information which is often crucial in molecular dynamics, for instance to compute the so-called transport coefficients.  A first attempt to address this  was made in~\cite{LegollLelievreOlla17}, where the authors prove trajectorial/pathwise estimates starting from a reversible dynamics. In~\cite{LelievreZhang18}, the results of~\cite{LegollLelievreOlla17} were generalized to the case of vectorial and nonlinear reaction coordinate.
  
In various situations, one has to introduce non-gradient force fields in the Langevin or overdamped Langevin dynamics. Examples include non-reversible diffusions introduced to speed up convergence to equilibrium and in variance reduction techniques~\cite{HwangHwangSheu93,DuncanLelievrePavliotis16}, and stochastic differential equations (SDEs) coupled with macroscopic fluid equations used for modelling polymer chains in a flow~\cite{Ottinger12,LeBrisLelievre12}. More generally, the class of non-equilibrium systems falls under this category.

The aim of this work is to study coarse-graining for non-reversible systems. Specifically, we construct effective dynamics starting from various cases of non-reversible SDEs  which admit an invariant measure. We then extend the ideas of~\cite{LegollLelievreOlla17} to the non-reversible setting, and thereby prove pathwise error estimates comparing the slow variables to the effective dynamics.  

\subsection{Central question}\label{Intro-sec:Quest}

We consider here various cases of non-reversible SDEs which have an invariant measure $\mu\in\mathcal P(\R^d)$. To start with, we focus on the simple case of non-gradient drift and identity diffusion matrix
\begin{align}\label{Intro:b}
dX_t=F(X_t) \, dt + \sqrt{2} \, dW_t, \quad X_{t=0}=X_0.
\end{align}
Here $X_t\in \R^d$ is the state of the system, $F:\R^d\rightarrow\R^d$ is a non-gradient function, 
$W_t$ is a $d$-dimensional Brownian motion and $X_0$ is the initial state of the system. 
 
In molecular dynamics one is often interested only in part of the state $X$. This is especially relevant from a numerical standpoint to reduce the computational complexity of the system. Therefore the aim is to study approximations of~\eqref{Intro:b} in the form of low-dimensional stochastic differential equations (SDEs). 

Such an approximation is made possible by a \emph{coarse-graining} map (also called a reaction coordinate) 
\begin{align*}
\xi:\R^d\rightarrow \mathcal M^k, \quad \text{with $k<d$},
\end{align*}
where the low-dimensional approximation is $t\mapsto \xi(X_t)$ and $\mathcal M^k$ is a $k$-dimensional smooth manifold. To avoid technical details, throughout this article we choose $\mathcal M=\R$, and for what follows next make the choice $\xi(x^1,\ldots,x^d)=x^1$ (see Remarks~\ref{rem:GenCG-A} and~\ref{rem:GenCG-B} below for generalisations to vector-valued affine coarse-graining maps). In this case the evolution of $\xi(X_t)=X^1_t$, called the \emph{projected dynamics}, is
\begin{align}\label{Intro:b-CG}
dX^1_t = F^1(X_t) \, dt +\sqrt{2} \, dB_t,
\end{align}
where $B_t=W^1_t$ is a $1$-dimensional Brownian motion. While~\eqref{Intro:b-CG} by construction is an \emph{exact} projection of $X_t$ under $\xi$, it is not closed and requires the knowledge of the full dynamics of $X_t$. As a result the projected dynamics~\eqref{Intro:b-CG} is as numerically intractable as the original dynamics~\eqref{Intro:b}. 

In~\cite{LL10} the authors construct an approximate equation by replacing the coefficients in~\eqref{Intro:b-CG} by their expectation with respect to the invariant measure $\mu$ of~\eqref{Intro:b}, conditioned on the value of $X^1_t$. This \emph{effective dynamics} reads
\begin{align}\label{Intro:b-Eff}
d Z_t = b(Z_t) \, dt + \sqrt{2} \, dB_t,
\end{align}
with the same Brownian motion and the initial condition $Z_0=X_0^1$ as in~\eqref{Intro:b-CG}. The effective drift $b:\R\rightarrow\R$ is given by 
\begin{equation}\label{Intro:def-Eff-drift}
b(z)=\E_\mu\pra{F^1(X) \, | \, X^1=z}.
\end{equation}
Here $\E_\mu$ is the expectation with respect to the measure $\mu$. The central advantage of using the effective dynamics over the projected dynamics is that it is Markovian and the effective drift can be calculated offline. A general property of the effective dynamics is that it admits $\xi_\#\mu$ (the image of the measure $\mu$ by $\xi$) as its invariant measure (see Remark~\ref{rem:Eff-Stat} below for details). Therefore it is at least consistent in terms of the stationary state. But we would like to study the dynamic consistency between~\eqref{Intro:b-CG} and~\eqref{Intro:b-Eff}. 

Using the effective dynamics~\eqref{Intro:b-Eff} instead of the projected dynamics~\eqref{Intro:b-CG} is justified only when $b(X^1_t)$ is a good approximation (in some sense) of $F^1(X_t)$. This of course happens when there is an inherent scale-separation present in the system so that, on the typical time scale of the slow variable $X^1_t$, the full dynamics $X_t$ samples the invariant measure on the level set $\{x \in \R^d; \ x^1=X^1_t\}$. However, as already noted in the introduction, in this article we work without explicit scale-separation, in the sense that we do not introduce an explicit parameter $\vep$ to accelerate the dynamics of the variables $(X^2_t,\ldots,X^d_t)$.

We now state the central question of this article: ``\emph{Can the pathwise difference between the projected dynamics~\eqref{Intro:b-CG} and the effective dynamics~\eqref{Intro:b-Eff} be quantified?}''. Before we give an overview of our results which address this question (see Section~\ref{Int-sec:MainRes}), we first motivate the study of the effective dynamics in the context of non-reversible systems based on statistical physics considerations. 

\subsection{Free-energy, mean force and non-equilibrium dynamics}\label{Int-sec:FreeEn}

This work is partly motivated by the following general question: is it possible to define a notion of free energy or mean force for non-reversible (namely non-equilibrium) systems? Let us explain the context.

For the reversible dynamics
$$
dX_t = - \nabla V(X_t) \, dt + \sqrt{2} \, dW_t
$$
which was considered in~\cite{LegollLelievreOlla17,LelievreZhang18} for example, the invariant measure is
$$
\mu(dx)= Z^{-1} \exp (-V(x)) \, dx
$$
where $\dps Z = \int_{\R^d} \exp(- V(x)) \, dx$ is assumed to be finite. Following the procedure described above, the effective dynamics is
$$
dZ_t = - A' (Z_t) \, dt + \sqrt{2} \, dB_t
$$
where $B_t=W^1_t$ and
$$
A'(z) = \E_\mu [ \partial_1 V (X) \, | \, X^1 = z]
$$
is the derivative of the so-called free energy $A$ defined by the relation (up to an irrelevant additive constant)
$$
\forall x^1\in\R, \quad \exp(-A(x^1)) = \int_{\R^{d-1}} \exp \left(-V(x^1,x^2,\ldots,x^d)\right) \, dx^2 \ldots dx^d.
$$
In this context, $-A'$ is called the mean force and it coincides with minus the derivative of the logarithm of $\xi_\#\mu$. The expression $\exp (- A(z)) \, dz=\xi_\#\mu$ is indeed the image of the measure $\mu$ by the coarse-graining map $\xi(x_1, \ldots,x_d)=x_1$. In such a context, it has been shown that the effective dynamics is a good one, in the sense that the time marginals of $Z_t$ are close to the time marginals of $X^1_t$ (see~\cite{LL10,DLPSS17}), and the trajectories of $Z_t$ are close to the trajectories of $X^1_t$ (see~\cite{LegollLelievreOlla17,LelievreZhang18}). In these works, the error is measured in terms of two quantities: (1)~a time-scale separation between $\xi(X_t)$ and the remaining coordinates, which is measured by the logarithmic-Sobolev inequality constant or the Poincar{\'e} inequality constant of the conditional measures $\mu(dx \, | \, \xi(x) = z)$ and (2)~the strength of the coupling between $\xi(X_t)$ and the remaininig coordinates, which is measured by a second-order cross derivative of the potential $V$. The question then arises, ``\emph{How can this be generalized to the non-reversible setting?}''. The purpose of this work is to define a proper notion of mean force for non-reversible dynamics. Let us conclude this section with two remarks. 

First, we emphasize that the notion of free energy is not properly defined for non-equilibrium systems, because the external force field always brings energy into the system since it is non-conservative. Therefore, in the context of non-reversible dynamics, it is more reasonable to focus on the notion of mean force (namely effective drift). The challenge then is to define an effective dynamics with an appropriate effective drift, which stays close to the original dynamics in some sense. 

Second, as an example, consider the simple toy problem on the 2-dimensional torus 
$$
dX_t = J \, \nabla \psi(X_t) \, dt + \sqrt{2} \, dW_t
$$
where $W_t$ is the 2-dimensional Brownian motion, $\dps J=\begin{pmatrix} 0 & 1 \\ -1 & 0 \end{pmatrix}$ is the canonical symplectic matrix and $\psi:\R^2\rightarrow\R$ is given by $\psi(x):=u\cdot x$ for a constant vector $u=(u_1,u_2)^T \in\R^2$. In this case, the invariant measure is the uniform law on the torus. If one considers the coarse-graining map $\xi(x_1,x_2)=x_1$, the image of the invariant measure under $\xi$ is the uniform law on the one-dimensional torus, and the derivative of the logarithm of its density is therefore zero. But it is easy to check that the effective dynamics
$$
d Z_t = \sqrt{2} \, dW^1_t
$$
is not close in any sense to $(X^1_t)_{t \ge 0}$, which follows the dynamics
$$
dX^1_t = u_2 \, dt + \sqrt{2} \, dW^1_t.
$$
Thus, taking minus the derivative of the logarithm of $\xi_\#\mu$ as the drift is not a good idea for non-reversible systems. Instead, following the procedure described above, let us consider
$$
dZ_t = b(Z_t) \, dt + \sqrt{2} \, dW^1_t,
$$
where the effective drift $b$ is given by~\eqref{Intro:def-Eff-drift}: 
$$
b(z) = \int_0^1 \partial_{x_2}\psi= u_2.
$$ 
Notice that this effective dynamics has the correct invariant measure (namely the uniform law on the torus), and is actually exact, i.e. $Z_t=X^1_t$ for any $t\geq 0$. This very simple example motivates the definition~\eqref{Intro:def-Eff-drift} of mean force (or effective drift) that we consider in this work. 

\subsection{Main results}\label{Int-sec:MainRes}

In this article, we consider the following SDE on $\R^d$:
\begin{align}\label{Intro:GenSDE}
dX_t = F(X_t) \, dt + \sqrt{2} \, \Sigma(X_t) \, dW_t, \quad X_{t=0}=X_0,
\end{align}
with a drift $F:\R^d\rightarrow\R^d$ and a non-degenerate   diffusion matrix $\Sigma:\R^d\rightarrow\R^{d\times d'}$. Here $W_t$ is the $d'$-dimensional Brownian motion and $X_0$ is the initial state. We assume that the coefficients satisfy conditions such that~\eqref{Intro:GenSDE} admits a unique strong solution and an invariant measure $\mu\in\mathcal P(\R^d)$. Unless explicitly stated, we choose the coarse-graining map to be a coordinate projection, i.e.
\begin{align*}
\xi(x^1,\ldots,x^d)=x^1.
\end{align*}
Based on the difference in the analysis involved, we study the dynamics~\eqref{Intro:GenSDE} in two cases:
\begin{enumerate}[topsep=0pt,wide=0pt, label=(\Alph*)]
\item
In Section~\ref{Sec-CaseA}, non-conservative drift $F$ and identity diffusion matrix, i.e. $\Sigma=\Id_d\in \R^{d\times d}$ (with $W_t$ a $d$-dimensional Brownian motion).
\item 
In Section~\ref{sec:NonRevSDE}, non-conservative drift $F$ and general diffusion matrix $\Sigma:\R^d\rightarrow\R^{d\times d'}$ (with $W_t$ a $d'$-dimensional Brownian motion).
\end{enumerate}
The choice of a scalar linear coarse-graining map is for technical convenience and appropriate generalisations are discussed in Remarks~\ref{rem:GenCG-A} and~\ref{rem:GenCG-B} below. 

We now state our main results for each of these two cases. We present two error estimates comparing the projected dynamics $X^1_t$ and the effective dynamics $Z_t$ in each of these cases: an easy-to-prove but weak estimate, and a more involved strong estimate on 
\begin{align*}
\E\pra{\sup\limits_{t\in [0,T]}\abs{X^1_t-Z_t}^2}.
\end{align*}

\subsubsection{Non-conservative drift and identity diffusion}

This case was introduced in Section~\ref{Intro-sec:Quest}, with the projected dynamics~\eqref{Intro:b-CG} and the effective dynamics~\eqref{Intro:b-Eff}-\eqref{Intro:def-Eff-drift}. Under the assumption that the effective drift $b$ is Lipschitz, we prove two error estimates (see Theorem~\ref{CaseA-thm:Strong}, Remark~\ref{rem:CaseA-lip} and Proposition~\ref{CaseA-lem:WeakEst} below for a precise statement): for every $T>0$, there exist constants $C_W$ and $C_S>0$ such that 
\begin{align*}
\text{(weak error estimate)} \qquad &\E\pra{\sup\limits_{t\in [0,T]}\abs{X^1_t-Z_t}^2} \leq C_W \frac{\kappa^2}{\alpha_{\PI}},\\
\text{(strong error estimate)}\qquad &\E\pra{\sup\limits_{t\in [0,T]}\abs{X^1_t-Z_t}^2} \leq C_S \frac{\kappa^2}{\alpha_{\PI}^2}.
\end{align*}
Here the constants $\kappa$ and $\alpha_{\PI}$ quantify the ``coupling'' between the dynamics of $X_t$ and the projected variable $X^1_t$. Specifically, $\alpha_{\PI}$ is the constant in the Poincar{\'e} inequality satisfied by the family of conditional invariant measures $\mu(\cdot \, | \, X^1=x^1)$. The constant $\kappa$ measures the strength of the interaction between the projected drift $F^1(X)$ and the fast variables $(X^2,\ldots, X^d)$. Though there is no explicit scale-separation present, the terminology of \emph{fast} and \emph{slow} variables is inspired by the fact that in practice the coarse-grained variable is typically the slow one. The constant $\alpha_{\PI}$ in the Poincar{\'e} inequality is a way to encode scale separation. Since we think of $\alpha_{\PI}$ as being a large constant, the strong estimate is indeed stronger than the weak estimate. The constants $C_W$ and $C_S$ do not depend on $\kappa$ and $\alpha_{\PI}$.

\subsubsection{Non-conservative drift $F$ and general $\Sigma$}\label{Int-sec:GenSDE-Res}

We now discuss the case of the general SDE~\eqref{Intro:GenSDE}. To present our main results we first need to construct the effective dynamics. The projected dynamics is  
\begin{align}\label{Int:full-CG}
dX^1_t=F^1(X_t) \, dt + \sqrt{2} \, |\Sigma^{1}|(X_t) \, dB_t,
\end{align}
where $\Sigma^1$ is the first row of the matrix $\Sigma$, $|\Sigma^1|^2:=\sum\limits_{j=1}^{d'} |\Sigma^{1j}|^2$ and $B_t$ is the one-dimensional Brownian motion defined by
\begin{align*}
dB_t=\sum_{j=1}^{d'} \frac{\Sigma^{1j}}{|\Sigma^{1}|}(X_t) \, dW^j_t.
\end{align*}
Here $W_t^j$ is the $j$-th component of $W_t$.
Following ideas discussed above, we construct the effective dynamics by averaging the coefficients of~\eqref{Int:full-CG} with respect to the invariant measure conditioned on the first variable, i.e.
\begin{align*}
d Z_t = b(Z_t) \, dt + \sqrt{2} \, \sigma(Z_t) \, dB_t,
\end{align*}
with the initial condition $Z_0=X^1_0$ and the coefficients $b,\sigma:\R\rightarrow\R$ defined by
\begin{align*}
b(z):=\E_{\mu} \left[ F^1(X) \, | \, X^1=z \right], \qquad \sigma^2(z)= \E_\mu\left[ |\Sigma^{1}|^2(X) \, | \, X^1=z \right].
\end{align*}
Under assumption made precise in Theorem~\ref{CaseC-thm:Strong} below, we prove two error estimates: for every $T>0$, there exist constants $C_W$ and $C_S>0$ such that
\begin{align*}
\text{(weak error estimate)} \qquad &\E\pra{\sup\limits_{t\in [0,T]}\abs{X^1_t- Z_t}^2} \leq C_W \bra{\frac{\kappa^2}{\alpha_{\PI}}+\frac{\lambda^2}{\alpha_{\PI}}},\\
\text{(strong error estimate)}\qquad &\E\pra{\sup\limits_{t\in [0,T]}\abs{X^1_t-Z_t}^2} \leq C_S \bra{\frac{\kappa^2}{\alpha^2_{\PI}} +\frac{\lambda^2}{\alpha_{\PI}}}.
\end{align*}
The constants $\kappa$ and $\lambda$ measure the strength of the interaction between the projected coefficients $F^1(X)$, $|\Sigma^{1}|(X)$ and the fast variables $(X^2,\ldots,X^d)$. As before, $\alpha_{\PI}$ is the constant in the Poincar{\'e} inequality satisfied by the family of conditional invariant measures $\mu(\cdot \, | \, X^1=x^1)$. The constants $C_W$ and $C_S$ do not depend on $\kappa$, $\lambda$ and $\alpha_{\PI}$. We refer to Proposition~\ref{CaseC-prop:Weak} and Theorem~\ref{CaseC-thm:Strong} below for a precise statement of the above estimates.

The crucial point to be noted is that, while the second estimate is sharper, in the regime of large $\alpha_{\PI}$ both the weak and strong error estimates scale as $O(1/\alpha_{\PI})$ in this setting (as opposed to the identity diffusion case). However, we remark that, in the specific case when the projected diffusion coefficient only depends on the slow variable, i.e. $|\Sigma^{1}|=|\Sigma^{1}|(x^1)$, then $\lambda=0$ and the strong error estimate is of the order of $1/\alpha^2_{\PI}$ (see Remark~\ref{rem:PepinSetup} below for details).

We also point out to Remark~\ref{rem:random_clock} below where a different strong estimate is proved (see~\eqref{eq:chaud13}), that may be useful in some situations (see e.g. Remark~\ref{rem:cas_eps}). 

Table~\ref{tab:err} summarizes our estimates in each of the cases.

\renewcommand{\arraystretch}{2.5}
\begin{table}[tbh]
\begin{center}
\begin{tabular}{|c|c|c|} 
\hline
& \textbf{Weak error estimate} & \textbf{Strong error estimate}\\
\hline
$F=-\nabla V$, \ $\Sigma=\Id$ & $C_{W}\dfrac{\kappa^2}{\alpha_{\PI}}$ & $C_{S}\dfrac{\kappa^2}{\alpha^2_{\PI}}$ \\
general $F$, \ $\Sigma=\Id$ & $C_{W}\dfrac{\kappa^2}{\alpha_{\PI}}$ & $C_{S}\dfrac{\kappa^2}{\alpha^2_{\PI}}$ \\
general $F$, \ $|\Sigma^{1}|=|\Sigma^{1}|(x^1)$ & $C_{W}\dfrac{\kappa^2}{\alpha_{\PI}}$ & $C_{S}\dfrac{\kappa^2}{\alpha^2_{\PI}}$ \\
general $F$ and $\Sigma$ & $C_{W}\bra{\dfrac{\kappa^2}{\alpha_{\PI}}+\dfrac{\lambda^2}{\alpha_{\PI}}}$ & $C_{S}\bra{\dfrac{\kappa^2}{\alpha^2_{\PI}}+\dfrac{\lambda^2}{\alpha_{\PI}}}$\\[0.1in]
\hline
\end{tabular}
\caption{This table summarizes the upper bound on the pathwise error estimates between the projected and the effective dynamics $\dps \E\left[\sup\limits_{t\in [0,T]}\abs{X^1_t-Z_t}^2\right]$, for various choices of coefficients in~\eqref{Intro:GenSDE}, under the assumption that the effective drift $b$ is Lipschitz. The weak and strong constants $C_W$ and $C_S$ do not depend on the Poincar{\'e} constant $\alpha_{\PI}$ and the coupling constants $\lambda$ and $\kappa$. The first row gives the results obtained in~\cite[Prop.~3 and Lem.~8]{LegollLelievreOlla17}. \label{tab:err}
}
\end{center}
\end{table}
\renewcommand{\arraystretch}{1}

\subsection{Central ingredients of the proof}\label{Intro-sec:WvS} 

As discussed in the previous section, in each of the cases studied we prove a weak and a strong error estimate. We now outline the crucial ideas and differences in the proofs of these estimates, in the simple case of the identity diffusion matrix discussed in Section~\ref{Intro-sec:Quest}. 

Recall that our aim is to estimate the trajectorial error between the dynamics~\eqref{Intro:b-CG} of the first coordinate $X^1_t$ and the effective dynamics~\eqref{Intro:b-Eff}. For a large class of initial conditions, it is sufficient to consider the case $X_0\sim\mu$ (see Remark~\ref{rem:Stat} below for details). Assuming that the effective drift $b$ is Lipschitz with constant $L_{b}$, we can write
\begin{align*}
\abs{X^1_t-Z_t} \leq \abs{\int_0^t\bra{F^1(X_s)- b(X^1_s)}ds}+\int_0^t \abs{ b(X^1_s)-b(Z_s)}ds 
\leq
\abs{\int_0^t f(X_s) \, ds}+L_{b}\int_0^t\abs{X^1_s-Z_s}ds,
\end{align*}
where $f:\R^d\rightarrow\R$ is defined by $f(x):=F^1(x)- b(x^1)$. Applying the Young's inequality we find
\begin{align}\label{Int-eq:Weak-Strong}
\E\pra{\sup\limits_{t\in [0,T]}\abs{X^1_t-Z_t}^2} \leq 2\E\pra{\sup\limits_{t\in [0,T]}\abs{\int_0^tf(X_s) \, ds}^2}+ 2L_{b}^2 \, T \, \E\pra{\int_0^T\abs{X^1_s-Z_s}^2ds}.
\end{align}
If we can appropriately estimate the first term on the right hand side, the required final estimate follows by applying standard Gronwall-type arguments. The difference in the weak and the strong estimate lies in the bound on this term. For the weak error estimate, we apply the Cauchy-Schwarz inequality and use the stationarity assumption $X_0\sim\mu$, which implies 
\begin{align*}
\E\pra{\sup\limits_{t\in [0,T]}\abs{\int_0^t f(X_s) \, ds}^2} \leq T \E\pra{\int_0^T \abs{f(X_s)}^2ds} = T^2 \int_{\R^d}|f(x)|^2 \, \mu(dx).
\end{align*}
Using a Poincar{\'e} inequality to control $\|f\|^2_{L^2(\mu)}$ (recall that the mean of $f$ vanishes, by definition of $b$) leads us to the weak error estimate. While easy to prove, this estimate does not capture classical averaging estimates even for reversible systems (see~\cite[Section 7]{LegollLelievreOlla17}).

We strengthen this estimate by directly estimating the first term of the right hand side of~\eqref{Int-eq:Weak-Strong} using a forward-backward martingale method by Lyons and Zhang (see~\cite{LyonsZhang94} and~\cite[Section 2.5]{KomorowskiOllaLandim12} for details). This technique allows us to prove an $H^{-1}$ bound of the type
\begin{align*}
\E\pra{\sup\limits_{t\in [0,T]}\abs{\int_0^tf(X_s) \, ds}^2} \leq C(T) \, \|f\|^2_{H^{-1}(\mu)},
\end{align*}
and the remaining difficulty lies in controlling the $H^{-1}$-term on the right hand side. Following this procedure we arrive at a stronger error estimate as compared to the simple estimate described above. For a detailed explanation of the steps involved, see the steps stated before Lemma~\ref{CaseA-lem:LyZh}.

While the general philosophy stays the same when dealing with the case of the general non-reversible SDE~\eqref{Intro:GenSDE}, the non-constant diffusion matrix adds technical complexity in the proof of the stronger estimate. For instance, in Lemma~\ref{CaseC-lem:ED1-L2} we prove a crucial inequality  on the difference  between the projected and the effective diffusion matrix.  Moreover, one has to use an appropriate Riemannian structure based on the scalar product $(u,v)_A=u^TAv$ where $A=\Sigma\Sigma^T$.

\subsection{Comparison with other works and outline of the article}

The novelty of this article lies in the following. 
\begin{enumerate}[topsep=0pt,wide=0pt]
\item\emph{Mean force for non-reversible systems.} As discussed in Section~\ref{Int-sec:FreeEn}, the notion of free-energy is not clear in the non-reversible setting, and therefore we focus on the mean force. The error estimates provided in this article indicate that the effective drift as defined by~\eqref{Intro:def-Eff-drift} is a good choice for the mean force for non-reversible dynamics. 

\item\emph{In comparison with existing literature.} Starting with the Langevin and the overdamped Langevin equations, estimates comparing time-marginals of the coarse-grained and the effective dynamics have been proved in~\cite{LL10,LegollLelievre12,DLPSS17}. Furthermore in~\cite{LegollLelievreOlla17,LelievreZhang18} pathwise estimates have been proved starting with reversible dynamics.  In this article we study the more general situations of non-reversible SDEs. This implies some additional difficulties, in particular when using the Lyons-Zhang result to estimate the $H^{-1}$ norm of $f$ (see Lemmas~\ref{CaseC-lem:ED1-L2} and~\ref{CaseC-lem:LZ} below). 

\item\emph{Quantitative estimates and scale-separation.} In our main results we present a weak and a strong error estimate. The strong error estimate can be used to recover classical averaging results under weaker assumptions as compared to the literature (see~\cite[Section 7]{LegollLelievreOlla17} and~\cite{Pepin17} for specific cases). Following these results, we expect that the proof strategy employed in this article can be used to prove error estimates for non-reversible SDEs in the presence of explicit scale-separation. An example of such a case is briefly discussed in Remark~\ref{rem:cas_eps} below. The general study of such systems however remains outside the scope of this article and is left to future work. 
\end{enumerate}

The article is organized as follows. In Section~\ref{Sec-CaseA} we derive error estimates starting from a non-reversible SDE with identity diffusion matrix. This is then generalized in Section~\ref{sec:NonRevSDE} to a general SDE with non-identity diffusion matrix. 
 
\section{Non-conservative drift and identity diffusion matrix}\label{Sec-CaseA}

This section deals with the case of non-gradient drift with identity diffusion matrix discussed in Section~\ref{Intro-sec:Quest}. In Section~\ref{A-sec:Setup} we present a few preliminaries and set up the system. The main results in this case are presented in Section~\ref{A-sec:MainRes} (see in particular Theorem~\ref{CaseA-thm:Strong}), and proved in Section~\ref{A-sec:Proof}. 
 
\subsection{Setup of the system}\label{A-sec:Setup}

Our aim is to study the following SDE, already introduced in Section~\ref{Intro-sec:Quest}:
\begin{align}\label{A-eq:Main-SDE}
dX_t=F(X_t) \, dt + \sqrt{2} \, dW_t, \quad X_{t=0}=X_0.
\end{align}

\begin{assume}\label{CaseA-ass:MainAss} Throughout this section, we assume that 
\begin{enumerate}[topsep=0pt,label=({A}\arabic*)]
\item\label{CaseA-InvMeas} (Invariant measure). The dynamics~\eqref{A-eq:Main-SDE} admits an invariant measure $\mu\in \mathcal P(\R^d)$ which has a density with respect to the Lebesgue measure on $\R^d$. We abuse notation and use $\mu$ for the density as well. Without loss of generality we can assume that $\mu$ is of the form
\begin{align}\label{A-def:Stat-GB}
\mu(dx) = Z^{-1} e^{-V(x)}dx,
\end{align}
where $Z$ is the normalization constant and $V$ satisfies $e^{-V} \in L^1(\R^d)$. This implies that there exists a vector field $c:\R^d\rightarrow \R^d$ such that 
\begin{align}\label{Int-eq:StatMeas-Cond}
F=-\nabla V + c \quad \text{with} \quad \div(c \, \mu)=0.
\end{align}
\item\label{CaseA-Reg} (Regularity). The drift $F$ belongs to $(C^\infty(\R^d))^d \cap (L^2(\mu))^d$.
\end{enumerate}
\end{assume}

The fact that $\mu$ defined in~\eqref{A-def:Stat-GB} is an invariant measure implies the condition~\eqref{Int-eq:StatMeas-Cond} follows by looking at the Fokker-Planck equation associated to~\eqref{A-eq:Main-SDE} and noting that 
\begin{align*}
0 =\div(-\mu F + \nabla \mu) = \div(\mu[-F + \nabla\log\mu]) = \div(\mu[-c+\nabla V -\nabla V ]) = -\div(c \, \mu),
\end{align*}
which yields~\eqref{Int-eq:StatMeas-Cond}.
  
Let us now recall the effective dynamics in this setting. With the choice of a scalar coordinate projection as coarse-graining map, i.e. $\xi:\R^d\rightarrow \R$ with $\xi(x^1,\ldots,x^d)=x^1$, the projected dynamics $\xi(X_t)=X^1_t$ is
\begin{align}\label{CaseA:CG}
dX^1_t=F^1(X_t) \, dt + \sqrt{2} \, dB_t,
\end{align}
where $B_t=W^1_t$ is a $1$-dimensional Brownian motion. The effective dynamics we introduce is
\begin{align}\label{CaseA:Eff}
dZ_t = b(Z_t) + \sqrt{2} \, dB_t,
\end{align}
with the effective drift $b:\R\rightarrow\R$ given by~\eqref{Intro:def-Eff-drift}, that is
\begin{align*}
b(x^1)=\int_{\R^{d-1}}F^1(x^1,x^2_d) \, \overline{\mu}_{x^1}(dx^2_d).
\end{align*}
Here we have used the notation $x^2_d:=(x^2,\ldots,x^d)$ and $\mu(\cdot \, | \, x^1) =: \overline{\mu}_{x^1}(\cdot)\in\mathcal P(\R^{d-1})$ for the conditional invariant measure conditioned on the first variable. Using the formula~\eqref{A-def:Stat-GB} for $\mu$, the effective drift can explicitly be written as 
\begin{align*}
b(x^1)=\dfrac{\displaystyle\int_{\R^{d-1}} F^1(x^1,x^2_d) \, e^{-V(x^1,x^2_d)} \, dx^2_d}{\displaystyle\int_{\R^{d-1}} e^{-V(x^1,x^2_d)} \, dx^2_d}.
\end{align*}
Throughout this section we assume that the projected and the effective dynamics have the same initial condition: $X^1_0=Z_0$.

We now introduce some notions which are used in the sequel. For a test function $h:\R^d\rightarrow \R$, the infinitesimal generator for the full dynamics~\eqref{A-eq:Main-SDE} is
\begin{align}\label{CaseA-def:Gen}
Lh=F\cdot \nabla h+ \Delta h,
\end{align}
and the corresponding adjoint operator $L^\star$ in $L^2(\mu)$ is 
\begin{align}\label{CaseA-def:TimRevGen}
L^\star h=F_{\mathrm R}\cdot \nabla h + \Delta h, \ \  \text{ where } \ \  F_{\mathrm R}:=-\nabla V -c.
\end{align}
Here $c$ and $V$ are as defined in Assumption~\ref{CaseA-ass:MainAss}, and the subscript in $F_{\mathrm R}$ is used to indicate that this is the drift for the \emph{time-reversed} diffusion process corresponding to~\eqref{CaseA-def:Gen} (see~\cite{Follmer85} and the proof of Lemma~\ref{CaseA-lem:LyZh} below for details). We also need the symmetric part $L_{\mathrm{sym}}$ of the generator $L$, which is given by
\begin{align}\label{CaseA-def:SymL}
L_{\sym}h:=\frac{1}{2} (L+L^\star) h = -\nabla V \cdot \nabla h + \Delta h.
\end{align}
Furthermore we define the family of operators indexed by $z\in \R$
\begin{align}\label{CaseA-def:SymLz}
(L^{z}_{\sym} h)(z,x^2_d) = -\widehat\nabla V(z,x^2_d) \cdot \widehat\nabla h(z,x^2_d)+\widehat\Delta h(z,x^2_d),
\end{align}
where $h:\R^{d}\rightarrow\R$ and the superscript on the differential operators indicates that these operators act on the variables $(x_2,\ldots,x_d)$, i.e.
\begin{align}\label{CaseA-def:LevGrad}
\widehat \nabla h := (\partial_2 h, \ldots,\partial_d h)^T \ \  \text{and} \ \  \widehat\Delta h:=\sum\limits_{i=2}^d \partial_{ii}h.
\end{align}
The operator $L^{x^1}_{\sym}$ can be interpreted as the projection of $L_{\sym}$ onto $(x^2,\ldots,x^d)$.

\subsection{Main result}\label{A-sec:MainRes}

We now state our central result comparing the projected dynamics~\eqref{CaseA:CG} and the effective dynamics~\eqref{CaseA:Eff}.

\begin{thm}\label{CaseA-thm:Strong}
In addition to~\ref{CaseA-InvMeas}-\ref{CaseA-Reg}, assume that
\begin{enumerate}[topsep=0pt,label=({B}\arabic*)]
\item \label{PE-ass:NonGrad-Equi-Init-Data} The system starts at equilibrium, i.e. $X_0\sim\mu$ and $Z_0=X^1_0$.
\item \label{PE-ass:NonGrad-Poincare} The family of conditional invariant measures $\overline{\mu}_{x^1}\in \mathcal P(\R^{d-1})$ satisfies a Poincar{\'e} inequality uniformly in $x^1$ with constant $\alpha_{\PI}$: for any $x^1 \in \R$ and $h\in H^1(\R^{d-1},\overline{\mu}_{x^1})$, we have
\begin{align}\label{PE-ass:NonGrad-PI}
\int_{\R^{d-1}} \left( h(u) - \int_{\R^{d-1}} h(u) \, \overline{\mu}_{x^1}(du) \right)^2 \overline{\mu}_{x^1}(du) \leq \frac{1}{\alpha_{\PI}} \int_{\R^{d-1}} \abs{\widehat\nabla h(u)}^2 \overline{\mu}_{x^1}(du).
\end{align}

\item \label{PE-ass:NonGrad-PE-kappa} The first component of the drift satisfies $\widehat\nabla F^1 \in L^2(\R^d,\mu)$ with
\begin{align*}
\kappa^2:=\int_{\R^d}\abs{\widehat\nabla F^1(x)}^2\mu(dx)<\infty.
\end{align*}

\item \label{PE-ass:NonGrad-PE-Effective-Lipschitz} The effective drift $b$ is one-sided Lipschitz with constant $L_{b}$, i.e.
\begin{align}\label{CaseA-ass:OSL}
\forall x,y\in \R, \quad (x-y)(b(x)- b(y)) \leq L_{b}|x-y|^2.
\end{align}
Furthermore the effective drift satisfies
\begin{align}\label{PE-ass:NonGrad-FullGron-Ass}
C_{b'}:=\int_{\R} \left( \sup\limits_{s\in [-|x^1|,|x^1|]} |b'(s)| \right) \, \widehat\mu(dx^1)<\infty,
\end{align}
where $\widehat\mu=\xi_\#\mu\in\mathcal P(\R)$ with $\xi(x^1,\ldots,x^d)=x^1$.
\end{enumerate}
Then, for any $T> 0$, there exists a constant $C=C(T,C_{b'},L_{b})$ such that   
\begin{align}\label{CaseA-eq:MainEst}
\E\pra{ \sup\limits_{t\in [0,T]} \abs{X^1_t -Z_t}} \leq C \frac{\kappa}{\alpha_{\PI}}. 
\end{align}
\end{thm}
The proof of this theorem is postponed until Section~\ref{A-sec:Proof}.

\begin{rem}
The constant $\kappa$ can be seen as a measure of the interaction between the dynamics along the coarse-grained variable $X^1_t$ and the dynamics on the level-set of the coarse-graining map, which is characterized by the gradient of $F^1$ with respect to the remaining variables. Assuming that $\kappa$ is finite implies that the dynamics of $X^1_t$ is sufficiently decoupled from the dynamics of $(X^2_t,\ldots,X^d_t)$. In particular, if the projected drift $F^1$ is independent of $x^2_d$, then the projected dynamics~\eqref{CaseA:CG} is closed and as a result the effective dynamics~\eqref{CaseA:Eff} is exact. This is seen from the estimate~\eqref{CaseA-eq:MainEst} since $\kappa=0$ in this case. 

The requirement that the effective drift is one-sided Lipschitz is not sufficient to prove a Gronwall-type estimate, and we additionally require~\eqref{PE-ass:NonGrad-FullGron-Ass} to complete the proof (see~\cite[Section 5]{LegollLelievreOlla17} for a detailed analysis). Note that, if the effective drift is assumed to be Lipschitz, then~\eqref{PE-ass:NonGrad-FullGron-Ass} is satisfied. 

The central assumption made in Theorem~\ref{CaseA-thm:Strong} (and throughout the rest of this article) is that the conditional invariant measure satisfies a Poincar{\'e} inequality. In recent years functional inequalities, such as the Log-Sobolev and the Poincar{\'e} inequalities, have been extensively used to quantify metastability for probability measures. Assuming that the conditional invariant measure satisfies such an inequality implies that the invariant measure conditioned on the level sets of the coarse-graining map (in our case $x\mapsto x^1$) is easy to sample from, which is akin to saying that there is no metastability on the level sets of the coarse-graining map. We refer to~\cite{LegollLelievreOlla17} for a more detailed discussion.
\end{rem}

\begin{rem} \label{rem:CaseA-lip}
If we replace in Theorem~\ref{CaseA-thm:Strong} the Assumption~\ref{PE-ass:NonGrad-PE-Effective-Lipschitz} by the assumption that $b$ is Lipschitz with constant $L_b$, then we can prove an error estimate in the $L^2$ norm rather than the $L^1$ norm as in~\eqref{CaseA-eq:MainEst}: under these assumptions, we have
\begin{align}
\E\pra{ \sup\limits_{t\in [0,T]} \abs{X^1_t -Z_t}^2} \leq 27 \, \frac{\kappa^2}{\alpha_{\PI}^2} \, T \, e^{L_{b}^2 T^2}. 
\end{align}
The proof is based on~\eqref{Int-eq:Weak-Strong}, estimate~\eqref{eq:chaud4} below and a Gronwall argument.
\end{rem}

\begin{rem}\label{rem:Stat}
The stationarity assumption $X_0\sim\mu$ simplifies the analysis, as it implies that $X_t\sim\mu$ for any $t>0$. However, using standard arguments (see for instance~\cite[Corollary 6]{LegollLelievreOlla17}), this assumption can easily be relaxed to allow for initial conditions which have an $L^\infty$-density with respect to the invariant measure, i.e. $d\mu_0/d\mu\in L^\infty(\R^d)$ where $\mu_0=\mathrm{law}(X_0)$. Recent results~\cite{Pepin17} suggest that this can be generalized to an even wider class of initial data, but this is outside the scope of this article.
\end{rem}

Before we prove Theorem~\ref{CaseA-thm:Strong}, we first state a weaker result (see Proposition~\ref{CaseA-lem:WeakEst} below). To prove this result, we need the following estimate which controls the difference between the projected and the effective drift in $L^2(\mu)$.

\begin{lem}\label{CaseA-lem:fL2}
Assume that~\ref{PE-ass:NonGrad-Poincare} and~\ref{PE-ass:NonGrad-PE-kappa} hold. Then 
\begin{align*}
\int_{\R^d}\bra{F^1(x)- b(x^1)}^2\mu(dx) \leq \frac{\kappa^2}{\alpha_{\PI}}.
\end{align*}
\end{lem}

The proof follows by using the disintegration theorem along with the assumptions~\ref{PE-ass:NonGrad-Poincare}-\ref{PE-ass:NonGrad-PE-kappa}:
\begin{align*}
\int_{\R^d} \Bigl( F^1(x) - b(x^1)\Bigr)^2\mu(dx)
=
\int_{\R}\widehat\mu(dx^1) \int_{\R^{d-1}}\Bigl( F^1 - \int_{\R^{d-1}} F^1 \, \overline{\mu}_{x^1}\Bigr)^2 \, \overline{\mu}_{x^1} \leq \frac{1}{\alpha_{\PI}} \int_{\R^d} \left| \widehat \nabla F^1 \right|^2 \, \mu
=
\frac{\kappa^2}{\alpha_{\PI}},
\end{align*}
where $\widehat\mu = \xi_\#\mu \in \mathcal P(\R)$ is the marginal invariant measure corresponding to the coordinate projection onto the first variable $\xi(x)=x^1$.

\begin{prop}[Weak error estimate]\label{CaseA-lem:WeakEst}
In addition to~\ref{CaseA-InvMeas}-\ref{CaseA-Reg}, assume that~\ref{PE-ass:NonGrad-Equi-Init-Data}-\ref{PE-ass:NonGrad-PE-kappa} holds and that the effective drift is one-sided Lipschitz as in~\eqref{CaseA-ass:OSL}. Then 
\begin{align}\label{CaseA-eq:WeakEst}
\E\pra{\sup\limits_{t\in [0,T]}\abs{X^1_t-Z_t}^2} \leq \frac{e^{(2L_{b}+1)T} - 1}{2L_b+1} \ \frac{\kappa^2}{\alpha_{\PI}}.
\end{align}
\end{prop}

\begin{proof}
The proof of this result follows the ideas described in Section~\ref{Intro-sec:WvS}. Applying the It{\^o}'s lemma and using the one-sided Lipschitz property~\eqref{CaseA-ass:OSL} of the effective drift, we find
\begin{align*}
\frac{1}{2}(X^1_t-Z_t)^2
&= \int_0^t (X^1_s-Z_s) \, (b(X^1_s)- b(Z_s)) \, ds + \int_0^t (X^1_s-Z_s) \, f(X_s) \, ds
\\
&\leq \Bigl(\frac{1}{2} + L_{b} \Bigr) \int_0^t\abs{X^1_s-Z_s}^2ds + \frac{1}{2} \int_0^t |f(X_s)|^2ds,
\end{align*}
where $f:\R^d\rightarrow\R$ is defined by $f(x):=F^1(x)- b(x^1)$. Using Gronwall lemma, we find that, for any $t\in [0,T]$,
\begin{align*}
\abs{X^1_t-Z_t}^2 \leq \int_0^t e^{(2L_{b}+1)(t-s)} |f(X_s)|^2 \, ds.
\end{align*}
Using Assumption~\ref{PE-ass:NonGrad-Equi-Init-Data} which implies that $X_t\sim\mu$, we obtain
\begin{align*}
\E\pra{\sup\limits_{t\in [0,T]}\abs{X^1_t-Z_t}^2}
\leq
\E\pra{\int_0^T e^{(2L_{b}+1)(T-s)}\abs{f(X_s)}^2ds}
=
\frac{e^{(2L_{b}+1)T} - 1}{2L_b+1} \int_{\R^d}\abs{f(x)}^2\,\mu(dx).
\end{align*}
The estimate~\eqref{CaseA-eq:WeakEst} then follows by using Lemma~\ref{CaseA-lem:fL2}. 
\end{proof}

\begin{rem}[Improving error for monotonic effective drift]\label{CaseA-rem:Improve}
In~\eqref{CaseA-eq:WeakEst}, the error grows exponentially in time. This growth can be improved in certain situations. Consider the case when the effective drift is decreasing and its derivative is  bounded away from zero, i.e. for some $\lambda>0$
\begin{align*}
b'(z) \leq -\lambda.
\end{align*}
Using this assumption, we can write
\begin{align*}
\frac{1}{2}\abs{X^1_t-Z_t}^2
&=
\int_0^t(X^1_s-Z_s)(F^1(X_s)- b(X^1_s)) \, ds +\int_0^t(X^1_s-Z_s) (b(X^1_s)- b(Z_s)) \, ds
\\
& \leq
\frac{1}{2C}\int_0^t \abs{F^1(X_s)- b(X^1_s)}^2ds + \bra{\frac{C}{2}-\lambda}\int_0^t\abs{X^1_s-Z_s}^2ds
\end{align*}
for any $C>0$. Choosing $C<2\lambda$, we then get $\dps \abs{X^1_t-Z_t}^2 \leq C^{-1} \int_0^t \abs{F^1(X_s)- b(X^1_s)}^2ds$, and we obtain a linear growth in time as opposed to the exponential growth of Proposition~\ref{CaseA-lem:WeakEst}.
\end{rem}

\subsection{Proof of Theorem~\ref{CaseA-thm:Strong}}\label{A-sec:Proof}

This section is devoted to the proof of Theorem~\ref{CaseA-thm:Strong} using the following steps:
\begin{itemize}[topsep=0pt]
\item Using an argument due to Lyons and Zhang~\cite{LyonsZhang94}, for any $\Phi$, we estimate $\dps \E \left[ \sup\limits_{t\in [0,T]} \abs{\int_0^t \nabla^\star \Phi(X_s) \, ds}^2 \right]$ in terms of $\|\Phi\|_{L^2(\mu)}$, where $\nabla^\star$ is the dual of $\nabla$ in $L^2(\mu)$. This estimate (see Lemma~\ref{CaseA-lem:LyZh} below) relies on a forward-backward martingale argument and the introduction of a Poisson problem corresponding to the full dynamics.
\item Next we make an appropriate choice for $\Phi$ such that $\nabla^\star \Phi = f$. This requires to introduce an auxiliary Poisson problem on the variables $x^2_d$ (see Lemma~\ref{PE-prop:NonGrad-IdMob-PoisLevSet} below). 
\item Finally we use the Gronwall argument introduced in~\cite[Section 5]{LegollLelievreOlla17} for one-sided Lipschitz drifts to complete the proof.
\end{itemize}
We now proceed with each of these steps.

\begin{lem}\label{CaseA-lem:LyZh}
Let $(X_t)_{t\geq 0}$ be the solution to~\eqref{A-eq:Main-SDE} with initial condition distributed according to the equilibrium measure $\mu$ (see Assumption~\ref{PE-ass:NonGrad-Equi-Init-Data}). Consider a function $\Phi:\R^d\rightarrow\R^d$ such that $\Phi\in (C^\infty\cap L^2(\mu))^d$. Then, for any $T>0$, we have 
\begin{align}\label{CaseA-eq:NonGrad-LyonZhang-Res}
\E\Biggl[ \sup\limits_{t\in[0,T]} \, \Biggl| \int_0^t \nabla^\star \Phi(X_s) \, ds \Biggr|^2\Biggr] \leq \frac{27}{2} \, T \, \|\Phi\|^2_{L^2(\mu)},
\end{align}
where $\nabla^\star$ is the adjoint of the operator $\nabla$ with respect to the $L^2(\mu)$ inner product, and is explicitly given by
\begin{align}\label{def:nabla*}
\nabla^\star \Phi = \nabla V\cdot \Phi - \div\Phi.
\end{align}
\end{lem}

\begin{proof} The proof falls in two steps. 

\emph{Step 1.} For any $\eta>0$, consider the resolvent problem
\begin{align}\label{PE-eq:NonGrad-Res-Pb}
\eta w_\eta-L_{\sym}w_\eta=\nabla^\star \Phi,
\end{align}
where $L_{\sym}$ (defined in~\eqref{CaseA-def:SymL}) is the symmetric part of $L$. The corresponding variational problem is to find $w_\eta\in H^1(\mu)$ which solves 
\begin{align*}
\forall v\in H^1(\mu), \quad \eta\int_{\R^d} w_\eta \, v \, \mu +\int_{\R^d} \nabla w_\eta \cdot \nabla v \, \mu =-\int_{\R^d} \Phi \cdot \nabla v \, \mu.
\end{align*}
Using the Lax-Milgram theorem, this variational problem has a unique solution $w_\eta\in H^1(\mu)$ for any $\eta>0$. Choosing $v=w_\eta$, we obtain
\begin{align*}
\eta\|w_\eta\|^2_{L^2(\mu)} +\|\nabla w_\eta\|^2_{L^2(\mu)} \leq \|\Phi\|_{L^2(\mu)} \|\nabla w_\eta\|_{L^2(\mu)},
\end{align*}
from which we infer that
\begin{align}\label{PE-eq:NonGrad-Res-Est}
\forall \eta>0, \quad \|\nabla w_\eta\|_{L^2(\mu)} \leq \|\Phi\|_{L^2(\mu)} \ \ \text{ and } \ \ \sqrt{\eta} \, \|w_\eta\|_{L^2(\mu)}\leq \|\Phi\|_{L^2(\mu)},
\end{align}
which implies that
\begin{align}\label{PE-eq:NonGrad-L2Limit}
\lim\limits_{\eta\rightarrow 0} \eta \, \|w_\eta\|_{L^2(\mu)} =0.
\end{align}
Since $V$, $\mu$ and $\Phi$ are smooth, $w_\eta\in C^2$ by standard elliptic theory.

\emph{Step 2.} Since $w_\eta$ is sufficiently smooth, using It{\^o}'s lemma, we write that, for any $t\in [0,T]$, 
\begin{align}\label{PE-eq:NonGrad-Ito}
w_\eta(X_t)-w_\eta(X_0) =\int_0^t Lw_\eta(X_s) \, ds +\sqrt{2} \int_0^t \nabla w_\eta(X_s)\cdot dW_s.
\end{align}
For $s\in [0,T]$, we introduce the time-reversed process
\begin{align*}
Y_s:=X_{T-s}.
\end{align*}
Since the drift $F$ satisfies Assumption~\ref{CaseA-Reg} and $\mathrm{law}(X_t)=\mathrm{law}(X_0)=\mu$, we can use~\cite{Follmer85} to deduce the evolution of $Y_s$:
\begin{align*}
dY_s=F_{\mathrm R}(Y_s) \, ds+\sqrt{2} \, d\overline W_s,
\end{align*}
where $(\overline W_s)_{0\leq s\leq T}$ is a $d$-dimensional Brownian motion and where we recall that $F_{\mathrm R}=-\nabla V-c$. Note that the generator of the time-reversed process is the adjoint of $L$ in $L^2(\mu)$ (see~\eqref{CaseA-def:TimRevGen}). Applying It{\^o}'s lemma again, we write, for any $t\in [0,T]$, 
\begin{align}\label{CaseA-eq:NonGrad-Mob-Time-Rev-Ito}
w_\eta(Y_T)-w_\eta(Y_{T-t})=\int_{T-t}^T L^\star w_\eta(Y_s) \, ds + \sqrt{2} \, \int_{T-t}^T \nabla w_\eta(Y_s)\cdot d\overline W_s.
\end{align}
Setting 
\begin{align*}
M_t=\int_0^t\nabla w_\eta(X_s)\cdot dW_s \quad \text{ and } \quad \overline M_t=\int_{0}^t\nabla w_\eta(Y_s)\cdot  d\overline W_s,
\end{align*}
adding~\eqref{PE-eq:NonGrad-Ito} and~\eqref{CaseA-eq:NonGrad-Mob-Time-Rev-Ito} and using the definition~\eqref{CaseA-def:SymL} of $L_{\sym}$, we obtain
\begin{align}
0&= \int_0^t Lw_\eta(X_s) \, ds +\int_{T-t}^T L^\star w_\eta(Y_s) \, ds + \sqrt{2} \, (M_t+\overline M_T-\overline M_{T-t}) \nonumber\\
&= \int_0^t Lw_\eta(X_s) \, ds +\int_{0}^t L^\star w_\eta(Y_{T-s}) \, ds + \sqrt{2} \, (M_t+\overline M_T-\overline M_{T-t}) \nonumber\\
&= \int_0^t Lw_\eta(X_s) \, ds +\int_{0}^t L^\star w_\eta(X_s) \, ds + \sqrt{2} \, (M_t+\overline M_T-\overline M_{T-t}) \nonumber\\
&= 2\int_0^t L_{\sym}w_\eta(X_s) \, ds + \sqrt{2} \, (M_t+\overline M_T-\overline M_{T-t}). \label{PE-eq:NonGrad-LyonZhang}
\end{align}
The random process $(M_t)_{0\leq t\leq T}$ is a square-integrable martingale since 
\begin{align*}
\E\pra{\int_0^T|\nabla w_\eta |^2(X_s) \, ds } = \int_0^T\int_{\R^d}|\nabla w_\eta |^2(x) \, \mu(dx) \, ds = T \, \|\nabla w_\eta\|^2_{L^2(\mu)} \leq T \, \|\Phi\|^2_{L^2(\mu)},
\end{align*}
where the first equality follows since $X_s\sim\mu$ and the inequality follows from~\eqref{PE-eq:NonGrad-Res-Est}. Using Doob's inequality followed by It{\^o}'s isometry, we find
\begin{align*}
\E\pra{ \bra{ \sup\limits_{t\in [0,T]} |M_t|}^2 } \leq 4 \E \pra{ \int_0^T|\nabla w_\eta|^2(X_s) \, ds } \leq 4T \, \|\Phi\|^2_{L^2(\mu)}.
\end{align*}
Since 
$\mathrm{law}(X_{T-s})=\mathrm{law}(Y_s)=\mu$, a similar estimate holds for $(\overline M_{t})_{0\leq t\leq T}$. 

Applying the Young's inequality to~\eqref{PE-eq:NonGrad-LyonZhang} and inserting the above estimate, we find
\begin{multline} \label{eq:chaud2}
\E\pra{ \sup\limits_{t\in [0,T]}\abs{ \int_0^t L_{\sym}w_\eta(X_s) \, ds}^2}
\leq
\frac{3}{2}\bra{ \E\pra{ \sup\limits_{t\in [0,T]}|M_t|^2} + \E \left[ |\overline M_T|^2 \right] + \E\pra{ \sup\limits_{t\in [0,T]}|\overline M_t|^2} }
\\
\leq
\frac{27}{2} T \, \|\Phi\|^2_{L^2(\mu)}.
\end{multline}
We are now in position to bound the left hand side of~\eqref{CaseA-eq:NonGrad-LyonZhang-Res}. Using~\eqref{PE-eq:NonGrad-Res-Pb}, we have, for any $\nu>0$,
\begin{align*}
\abs{\int_0^t\nabla^\star \Phi(X_s) \, ds }^2 \leq (1+\nu) \abs{\int_0^tL_{\sym}w_\eta(X_s) \, ds }^2 + \bra{ 1+\frac{1}{\nu}} \abs{\int_0^t\eta w_\eta(X_s) \, ds }^2, 
\end{align*}
and therefore, using~\eqref{eq:chaud2},
\begin{align*}
\E\pra{ \sup\limits_{t\in [0,T]} \abs{\int_0^t\nabla^\star \Phi(X_s) \, ds }^2}
&\leq
(1+\nu)\frac{27}{2} T \, \|\Phi\|^2_{L^2(\mu)} + \bra{ 1+\frac{1}{\nu}} \E\pra{ \sup\limits_{t\in [0,T]}\abs{\int_0^t\eta w_\eta(X_s) \, ds }^2}
\\
&\leq
(1+\nu)\frac{27}{2} T \, \|\Phi\|^2_{L^2(\mu)} + \bra{ 1+\frac{1}{\nu}} \eta^2 T \, \E\pra{ \int_0^T w_\eta^2(X_s) \, ds}
\\
& \leq
(1+\nu)\frac{27}{2} T \, \|\Phi\|^2_{L^2(\mu)} + \bra{ 1+\frac{1}{\nu}} \eta^2 T^2 \, \|w_\eta\|^2_{L^2(\mu)},
\end{align*}
where the second inequality follows from the Cauchy-Schwarz inequality. Taking the limit $\eta\rightarrow 0$, using~\eqref{PE-eq:NonGrad-L2Limit} and next taking the limit $\nu\rightarrow 0$, we obtain the estimate~\eqref{CaseA-eq:NonGrad-LyonZhang-Res}. This concludes the proof of Lemma~\ref{CaseA-lem:LyZh}.
\end{proof}

The following lemma discusses the properties of an auxiliary Poisson problem which is useful to find some $\Phi$ such that $\nabla^\star \Phi =f$. For this result, we define the function space
\begin{align*}
H^1_m(\overline{\mu}_z) := \left\{v\in H^1(\overline{\mu}_z), \ \ \int_{\R^{d-1}}v\, \overline{\mu}_z=0 \right\}.
\end{align*}

\begin{lem}[Poisson problem on the level sets]\label{PE-prop:NonGrad-IdMob-PoisLevSet}
Assume that~\ref{PE-ass:NonGrad-Poincare} holds. Consider some function $\ell \in C^\infty(\R^d;\R)$ with $\ell(z,\cdot) \in L^2(\overline{\mu}_z)$ and $\dps \int_{\R^{d-1}} \ell(z,x^2_d) \, \overline{\mu}_z(dx^2_d)=0$ for any $z\in\R$. Then, for any $z\in \R$, there exists a unique solution $x^2_d\mapsto u(z,x^2_d)\in H^1_m(\overline{\mu}_z)$ to 
\begin{equation}\label{PE-eq:NonGrad-Pois-LevSet}
-(L^z_{\sym})u=\ell(z,\cdot) \ \ \text{ with } \ \ \int_{\R^{d-1}}u(z,x^2_d) \, \overline{\mu}_z(dx^2_d)=0,
\end{equation}
where we recall that $L^z_{\sym}$ is defined by~\eqref{CaseA-def:SymLz}. Furthermore, $u$ is a smooth function which satisfies
\begin{align}\label{PE-eq:NonGrad-Est-LevSet-Grad}
\forall z \in \R, \quad \left\| \widehat\nabla u \right\|^2_{L^2(\overline{\mu}_z)} \leq \frac{1}{\alpha_{\PI}}\|\ell(z,\cdot)\|^2_{L^2(\overline{\mu}_z)}.  
\end{align}
\end{lem}

\begin{proof}
Using the definition~\eqref{CaseA-def:SymLz} of $L^z_{\sym}$, we have, for any sufficiently smooth $u:\R^d\rightarrow\R$ and $v:\R^{d-1}\rightarrow\R$,
\begin{align*}
-\int_{\R^{d-1}}(L^z_{\sym})u \, v \, \overline{\mu}_z = \int_{\R^{d-1}} \widehat\nabla u(z,\cdot) \cdot \widehat \nabla v \, \overline{\mu}_z.
\end{align*}
Therefore, the variational problem corresponding to~\eqref{PE-eq:NonGrad-Pois-LevSet} is to find $u(z,\cdot) \in H^1_m(\overline{\mu}_z)$ which solves 
\begin{align} \label{eq:chaud3}
\forall v\in H^1_m(\overline{\mu}_z), \quad \int_{\R^{d-1}}\widehat\nabla u(z,\cdot)\cdot\widehat\nabla v\, \overline{\mu}_z = \int_{\R^{d-1}} \ell(z,\cdot) \, v \, \overline{\mu}_z.
\end{align}
The coercivity of the bounded bilinear form on the left hand side follows from Assumption~\ref{PE-ass:NonGrad-Poincare}. The above right hand side is well-defined since, for any $v\in H^1_m(\overline{\mu}_z)$ and $\ell\in L^2(\overline{\mu}_z)$, we have $\dps \left| (\ell,v)_{\overline{\mu}_z} \right| \leq \|\ell\|_{L^2(\overline{\mu}_z)} \|v\|_{L^2(\overline{\mu}_z)}\leq C$. Using the Lax-Milgram theorem, the variational problem~\eqref{eq:chaud3} therefore admits a unique solution.

Since $\ell$ and $V$ are smooth, it follows from standard elliptic theory that $u(z,\cdot)$ is smooth for any $z\in\R$.
Choosing $v=u(z,\cdot)$ in~\eqref{eq:chaud3} and using~\eqref{PE-ass:NonGrad-PI}, we get
\begin{multline*}
\int_{\R^{d-1}} \left| \widehat\nabla u(z,\cdot) \right|^2 \, \overline{\mu}_z
=
\int_{\R^{d-1}} \ell(z,\cdot) \, u(z,\cdot) \, \overline{\mu}_z
\leq
\|\ell(z,\cdot)\|_{L^2(\overline{\mu}_z)} \|u(z,\cdot)\|_{L^2(\overline{\mu}_z)}
\\
\leq
\frac{1}{\sqrt{\alpha_{\PI}}} \|\ell(z,\cdot)\|_{L^2(\overline{\mu}_z)} \left\| \widehat\nabla u(z,\cdot) \right\|_{L^2(\overline{\mu}_z)},
\end{multline*}
and therefore~\eqref{PE-eq:NonGrad-Est-LevSet-Grad} follows. 
\end{proof}

We now combine Lemmas~\ref{CaseA-lem:LyZh} and~\ref{PE-prop:NonGrad-IdMob-PoisLevSet} to prove Theorem~\ref{CaseA-thm:Strong}. 

\begin{proof}[Proof of Theorem~\ref{CaseA-thm:Strong}]
We wish to compare $X^1_t$ and $Z_t$ which respectively solve
\begin{align*}
dX^1_t &= b(X^1_t) \, dt + \sqrt{2} \, dB_t + f(X_t) \, dt,\\
dZ_t &= b(Z_t) \, dt + \sqrt{2} \, dB_t,
\end{align*}
with $f(x)=F^1(x)- b(x^1)$. Under Assumption~\ref{PE-ass:NonGrad-PE-Effective-Lipschitz}, we are in position to use~\cite[Lemma 12]{LegollLelievreOlla17} and we thus get that
\begin{align*}
\E\pra{ \sup\limits_{t\in [0,T]}\abs{ X^1_t -Z_t}} \leq C(T,C_{b'},L_{b}) \bra{\E\pra{ \sup_{t\in [0,T]}\abs{\int_0^t f(X_s) \, ds}^2}}^{1/2}.
\end{align*}
In what follows, we show that 
\begin{align} \label{eq:chaud4}
\E\pra{ \sup\limits_{t\in[0,T]}\abs{ \int_0^t f(X_s) \, ds}^2} \leq \frac{27}{2}T\frac{\kappa^2}{\alpha_{\PI}^2},
\end{align}
which yields~\eqref{CaseA-eq:MainEst}.

By definition of $b$ (see~\eqref{Intro:def-Eff-drift}), the function $f$ belongs to $C^\infty(\R^d;\R)$ with $\dps \int_{\R^{d-1}} f(z,\cdot) \, \overline{\mu}_{z}=0$ for any $z \in \R$. Using Lemma~\ref{CaseA-lem:fL2}, we additionally have $f(z,\cdot) \in L^2(\overline{\mu}_{z})$. Therefore we can make the choice $\ell=f$ in Lemma~\ref{PE-prop:NonGrad-IdMob-PoisLevSet}, which implies that there exists a unique solution to 
\begin{align}\label{PE-eq:NonGradLevelSet-f-Prob}
-(L^{x^1}_{\sym})u = f(x^1,\cdot) \ \ \text{ with } \ \ \int_{\R^{d-1}}u(x^1,x^2_d) \, \overline{\mu}_{x^1}(dx^2_d) = 0,
\end{align}
in the sense of Lemma~\ref{PE-prop:NonGrad-IdMob-PoisLevSet}. Integrating~\eqref{PE-eq:NonGrad-Est-LevSet-Grad} with respect to $\widehat \mu(dz)$ and using Lemma~\ref{CaseA-lem:fL2}, we obtain that
\begin{align}\label{PE-eq:NonGrad-L2-Grad}
\left\| \widehat\nabla u \right\|^2_{L^2(\mu)} \leq \frac{1}{\alpha_{\PI}}\|f\|^2_{L^2(\mu)}\leq \frac{\kappa^2}{\alpha_{\PI}^2}.
\end{align}
Next, we make the choice $\dps \Phi= \left( 0,\widehat\nabla u \right)^T:\R^d\rightarrow\R^d$ in Lemma~\ref{CaseA-lem:LyZh}, which is a valid choice since $u$ is smooth and $\widehat\nabla u \in (L^2(\mu))^d$ by~\eqref{PE-eq:NonGrad-L2-Grad}. In this case, using~\eqref{PE-eq:NonGradLevelSet-f-Prob}, we compute that
\begin{align*}
-\nabla^\star \Phi= -\widehat\nabla u \cdot\widehat\nabla V + \widehat\Delta u = (L^{x^1}_{\sym})u = - f(x^1,\cdot),
\end{align*}
where $L^z_{\sym}$ is defined in~\eqref{CaseA-def:SymLz} and $\nabla^\star$ is defined in~\eqref{def:nabla*}. Therefore, with this choice in~\eqref{CaseA-eq:NonGrad-LyonZhang-Res} and using~\eqref{PE-eq:NonGrad-L2-Grad}, we get
\begin{align*}
\E\Biggl[ \sup\limits_{t\in[0,T]}\Biggl| \int_0^t f(X_s) \, ds \Biggr|^2\Biggr] \leq \frac{27}{2}T \, \left\| \widehat\nabla u \right\|^2_{L^2(\mu)} \leq \frac{27}{2}T \, \frac{\kappa^2}{\alpha_{\PI}^2}.
\end{align*}
This completes the proof of the claim~\eqref{eq:chaud4} and thus of Theorem~\ref{CaseA-thm:Strong}.
\end{proof}

So far we have discussed bounds on the pathwise error between the projected and the effective dynamics. In the following remark we present error estimates on time marginals. 

\begin{rem}[Time-marginal estimates] 
Theorem~\ref{CaseA-thm:Strong} estimates the pathwise error between the projected and the effective dynamics. Using ideas of~\cite{LL10,DLPSS17}, we can also estimate the time-marginal error in relative entropy. We assume that $\mathrm{law}(X^1_0)=\mathrm{law}(Z_0)$ and that the family of conditional invariant measures $\overline{\mu}_{x^1}$ satisfy a Log-Sobolev inequality with constant $\alpha_{\LSI}$ (as opposed to the weaker Poincar{\'e} inequality assumed in Theorem~\ref{CaseA-thm:Strong}). Then, repeating the arguments as in~\cite{LL10,DLPSS17}, we obtain that there exists a constant $C$ (which only depends on the $L^\infty$ norm of $\widehat \nabla F^1$) such that
\begin{align*}
\forall t>0, \quad \RelEnt \left( \mathrm{law}(X^1_t) \, | \, \mathrm{law}(Z_t) \right) \leq \frac{C}{\alpha_{\LSI}^2} \Big( \RelEnt\left( \mathrm{law}(X_0) \, | \, \mu \right) - \RelEnt\left(\mathrm{law}(X_t) \, | \, \mu\right) \Big),
\end{align*}
where, for any two probability measures $\zeta,\eta\in\mathcal P(\R)$ with $f:=d\zeta/d\eta$, the relative entropy is defined by $\dps \RelEnt(\zeta \, | \, \eta) := \int_{\R} f \, \log f \, d\eta$. This gives another way to assess the accuracy of the effective dynamics. 
\end{rem}

Until now, we have focussed on SDEs with scalar coordinate projections as the coarse-graining mapping. In the following remark we discuss generalisations to vector-valued affine coarse-graining maps. 

\begin{rem}\label{rem:GenCG-A}
Consider the SDE~\eqref{A-eq:Main-SDE}, that is
\begin{align*}
dX_t=F(X_t) \, dt + \sqrt{2} \, dW_t, \quad X_{t=0}=X_0,
\end{align*}
and the case of a {\em vectorial} affine coarse-graining map $\xi:\R^d\rightarrow\R^k$ with $\xi(x)=\mathbb Tx+\tau$, where $\tau\in\R^k$ and $\mathbb T\in\R^{k\times d}$ has full rank. 
With the notation $\widehat X_t:=\xi(X_t)$ and under the convention $\dps [\nabla \xi]_{ij} = \frac{\partial \xi_i}{\partial x_j}$ for any $1 \leq i \leq k$ and $1 \leq j \leq d$, the projected dynamics is
\begin{align*}
d\widehat X_t=(\nabla\xi \, F)(X_t) \, dt + \sqrt{2} \, \widehat\Sigma \, dB_t, 
\end{align*}
where the diffusion matrix $\widehat\Sigma^2:=\nabla\xi\nabla\xi^T=\mathbb T\mathbb T^T$ is a constant $R^{k\times k}$ matrix and $B_t$ is a $k$-dimensional Brownian motion given by $\dps dB_t = (\widehat\Sigma)^{-1} \, \nabla\xi \, dW_t$.

The corresponding effective dynamics is
\begin{align*}
dZ_t = b(Z_t) \, dt + \sqrt{2} \, \widehat\Sigma \, dB_t,
\end{align*}
with the effective drift $b:\R^k\rightarrow\R^k$ given by (compare with~\eqref{Intro:def-Eff-drift})
\begin{align*}
\forall z\in\R^k, \quad b(z):=\int_{\xi^{-1}(z)} (\nabla\xi \, F)(y) \, \overline{\mu}_{z}(dy).
\end{align*}
As before $\mu(\cdot \, | \, \xi(x)=z) =: \overline{\mu}_z(\cdot) \in \mathcal P(\xi^{-1}(z))$ is the family of conditional invariant measures associated with $\mu$ and $\xi$. We assume that
\begin{enumerate}
\item The system starts at equilibrium, i.e. $X_0\sim\mu$.
\item The family of measures $\overline{\mu}_{z}\in \mathcal P(\xi^{-1}(z))$ satisfies a Poincar{\'e} inequality uniformly in $z\in\R^k$ with constant $\alpha_{\PI}$, in the sense that, for any $z\in\R^k$ and $h:\xi^{-1}(z)\rightarrow\R$ for which the right hand side below is finite, we have
$$
\int_{\xi^{-1}(z)} \Bigr(h-\int_{\xi^{-1}(z)}h \, \overline{\mu}_z(dy) \Bigl)^2 \, \overline{\mu}_z(dy) \leq \frac{1}{\alpha_{\PI}} \int_{\xi^{-1}(z)} \abs{\nabla_z h}^2 \, \overline{\mu}_z(dy),
$$
where $\nabla_z$ is the surface gradient on the level set $\xi^{-1}(z)$ given by
\begin{equation}\label{surf:grad}
\nabla_z := \left( \Id_d - \sum_{1 \leq i,j \leq k} (\widehat\Sigma^{-1})_{ij} \nabla\xi_i \otimes \nabla\xi_j \right) \nabla.
\end{equation}
\item The effective drift $b$ is Lipschitz with constant $L_{b}$.
\item The projected drift satisfies $\nabla_{z} (\nabla \xi \, F)_i \in L^2(\R^d,\mu)$ for any $1 \leq i \leq k$, with
$$
\kappa^2 := \sum_{i=1}^k \int_{\R^d} \abs{\nabla_z(\nabla\xi \, F)_i }^2 \, \mu<\infty.
$$
\end{enumerate}
Then we have
\begin{align}\label{CaseA-eq:MainEst_bis}
\E\pra{ \sup\limits_{t\in [0,T]}\abs{ X^1_t -Z_t}^2} \leq C(T,L_{b}) \, \frac{\kappa^2}{\alpha^2_{\PI}}. 
\end{align}
The proof of~\eqref{CaseA-eq:MainEst_bis} is a straightforward generalisation of the proof of Theorem~\ref{CaseA-thm:Strong}, using Remark~\ref{rem:CaseA-lip}.
\end{rem}

\section{Non-conservative drift and general diffusion matrix}\label{sec:NonRevSDE}

In this section we focus on the SDE
\begin{align}\label{CaseC-eq:SDE}
dX_t=F(X_t) \, dt +\sqrt{2} \, \Sigma(X_t) \, dW_t, \quad X_{t=0}=X_0.
\end{align}
Here the drift $F$ is a function from $\R^d$ to $\R^d$, the non-identity diffusion matrix $\Sigma$ is a function from $\R^d$ to $\R^{d\times d'}$ and $W_t$ is a $d'$-dimensional Brownian motion. In what follows, we use the notation
\begin{align} \label{eq:defA}
A := \Sigma\Sigma^T \in \R^{d \times d}.
\end{align}
In Section~\ref{CaseC-setup}, we set up the system. The main results are presented in Section~\ref{CaseC-sec:MainRes} and proved in Section~\ref{CaseC-sec:Proof}. 

\subsection{Setup of the system}\label{CaseC-setup}

\begin{assume}\label{CaseB-ass:MainAss} Throughout this section we assume that 
\begin{enumerate}[topsep=0pt,label=({C}\arabic*)]
\item \label{CaseB-ass:Stat} (Invariant measure). The SDE~\eqref{CaseC-eq:SDE} admits an invariant measure $\mu\in\mathcal P(\R^d)$ which has a density with respect to the Lebesgue measure on $\R^d$. As in Section~\ref{A-sec:Setup}, we abuse notation and denote the density by $\mu$. Without loss of generality we can assume that $\mu$ is of the form
\begin{align}\label{B-def:Stat-GB}
\mu(dx) = Z^{-1} e^{-V(x)}dx,
\end{align}
where $Z$ is the normalization constant and $V$ satisfies $e^{-V}\in L^1(\R^d)$. This implies that there exists a vector field $c:\R^d\rightarrow \R^d$ such that 
\begin{align}\label{B-eq:StatMeas-Cond}
F=-A\nabla V+\div A+c \quad \text{with} \quad \div(c \, \mu)=0. 
\end{align}

\item \label{CaseA-ass:b-Reg2} (Regularity). The coefficients $F$ and $\Sigma\in C^\infty$ are Lipschitz. Furthermore, for any $x \in \R^d$, the diffusion matrix $A(x)=\Sigma(x)\Sigma^T(x)$ is positive-definite. 
\end{enumerate}
\end{assume}

The fact that~\eqref{B-def:Stat-GB} implies~\eqref{B-eq:StatMeas-Cond} follows by studying the Fokker-Planck equation associated to~\eqref{CaseC-eq:SDE} and noting (using the Einstein's summation notation) that
\begin{multline*}
0
=
\partial_i \big(-\mu F^i + \partial_j (A^{ij} \mu) \big)
=
\partial_i \big( \mu[-F^i+\partial_j A^{ij} +A^{ij}\partial_j\log\mu] \big)
\\=
\partial_i \big( \mu[A^{ij}\partial_j V-\partial_j A^{ij}-c^i+\partial_j A^{ij} +A^{ij}\partial_j\log\mu] \big)
=
-\partial_i\bra{c^i \, \mu}.
\end{multline*}

The projected variable $X^1_t$ satisfies   
\begin{align}\label{CaseC-eq:Pro}
dX^1_t=F^1(X_t) \, dt+\sqrt{2} \, |\Sigma^{1}|(X_t) \, dB_t,
\end{align}
where $|\Sigma^1|^2:=\sum\limits_{j=1}^{d'} |\Sigma^{1j}|^2$ and $B_t$ is the one-dimensional Brownian motion defined by
\begin{align*}
dB_t=\sum_{j=1}^{d'}\frac{\Sigma^{1j}}{|\Sigma^{1}|}(X_t) \ dW^j_t.
\end{align*}
Here $W_t^j$ is the $j$-th component of $W_t$. Using the same approach as in Section~\ref{Sec-CaseA}, we build the effective dynamics by taking the conditional expectation of the projected coefficients with respect to the invariant measure. We thus define the effective dynamics as
\begin{align}\label{CaseC-def:ED1}
dZ_t= b(Z_t) \, dt+\sqrt{2} \, \sigma(Z_t) \, dB_t,
\end{align}
with coefficients $b, \sigma:\R\rightarrow\R$ given by
\begin{align}\label{CaseC-E1:Coeff}
b(x^1):=\int_{\R^{d-1}} F^1(x^1,x^2_d) \, \overline{\mu}_{x^1}(dx^2_d) \ \ \text{ and } \ \ \sigma^2(x^1):=\int_{\R^{d-1}} |\Sigma^{1}|^2(x^1,x^2_d) \, \overline{\mu}_{x^1}(dx^2_d),
\end{align}
where $\overline{\mu}_{x^1}\in\mathcal P(\R^{d-1})$ is the family of conditional measures conditioned on the first variable. Throughout this section, we assume that the projected and the effective dynamics have the same initial condition: $X^1_0=Z_0$.

In the following remark we discuss the invariant measure of the effective dynamics~\eqref{CaseC-def:ED1}.

\begin{rem}[Stationary measure of the effective dynamics]\label{rem:Eff-Stat}
Consider the coarse-graining map $\xi(x^1,\ldots,x^d)=x^1$ and let $\xi_\#\mu\in\mathcal P(\R)$ be the push-forward of the invariant measure under $\xi$. We show that $\xi_\#\mu$ is the invariant measure of the effective dynamics, i.e. $Z_t\sim\xi_\#\mu$ for any $t$ if $Z_0\sim\xi_\#\mu$. To show this, we define $\zeta_t:=\mathrm{law}(\xi(X_t))$, which evolves according to (see~\cite{Gyongy86})
\begin{align*}
\partial_t\zeta=\partial_z \left(\widehat b \, \zeta \right)+\partial_{z}^2 \left(\widehat \sigma^2 \, \zeta\right),
\end{align*}
with coefficients 
\begin{align*}
\forall t\geq 0, \ \forall x^1\in\R, \quad \widehat b(t,x^1)=\E \left[ F^1(X_t) \, | \, \xi(X_t)=x^1 \right] \quad \text{ and } \quad \widehat\sigma^2(t,x^1):=\E\left[|\Sigma^1|^2(X_t) \, | \, \xi(X_t)=x^1 \right].
\end{align*}
Assume that the full system starts at equilibrium: $X_0\sim\mu$, and therefore $X_t\sim\mu$. As a result, $\widehat b(t,x^1)=b(x^1)$ and $\widehat\sigma^2(t,x^1)=\sigma^2(x^1)$, where $b$ and $\sigma^2$ are defined by~\eqref{CaseC-E1:Coeff}. This implies that $\zeta_t$ and the law of $Z_t$ satisfy the same Fokker-Planck equation, with the same initial condition, and thus that, for any $t$, $\mathrm{law}(Z_t)=\zeta_t$. In addition, since $X_t\sim\mu$, we have $\xi(X_t)\sim\xi_\#\mu$ for any $t$ and thus $\zeta_t = \xi_\#\mu$. We hence get that $\mathrm{law}(Z_t) = \xi_\#\mu$ for any time $t$. Using this argument it follows that, starting with an SDE with invariant measure $\mu$ and any (possibly nonlinear) coarse-graining map $\xi$, the corresponding effective dynamics has $\xi_\#\mu$ as an invariant measure.
\end{rem}

We now introduce some notions that we need to prove our main results. For a test function $h:\R^d\rightarrow \R$, the generator for the full dynamics~\eqref{CaseC-eq:SDE} is
\begin{align*}
Lh=F\cdot \nabla h+ A:\nabla^2 h,
\end{align*}
where $\nabla^2$ is the Hessian, $M:N=\mathrm{tr}(M^T N)$ is the Frobenius inner product for matrices, and $A$ is defined by~\eqref{eq:defA}. The following lemma characterizes the adjoint operator of $L$ in $L^2(\mu)$.

\begin{lem}\label{CaseB-lem:Adj}
The adjoint operator $L^\star$ of $L$ in $L^2(\mu)$ is 
\begin{align*}
L^\star h=F_{\mathrm R}\cdot \nabla h + A:\nabla^2h,
\end{align*}
with $F_{\mathrm R}:=-A\nabla V+\div A-c$, where $c$ is defined in~\eqref{B-eq:StatMeas-Cond}.
\end{lem}

\begin{proof}
Using the characterisation~\eqref{B-eq:StatMeas-Cond} and integration by parts, we write, for any test functions $u$ and $v$, that
\begin{align*}
(Lu,v)_{L^2(\mu)}
&=\int \big( F^i\partial_i u + A^{ij}\partial_{ij} u \big) \, v \, \mu
\\
&= \int \partial_iu \, \Big( (\partial_j A^{ij}) \, v \, \mu - A^{ij} (\partial_jV) \, v \, \mu + c^i \, v \, \mu - (\partial_j A^{ij}) \, v \, \mu - A^{ij} (\partial_jv) \, \mu + A^{ij} (\partial_jV) \, v \, \mu \Big)
\\
&=\int \partial_i u \, \big( c^i \, v \, \mu - A^{ij} (\partial_j v) \, \mu \big)
\\
&=
\int u \Big( -\partial_i (c^i\mu) \, v - c^i \, (\partial_iv) \, \mu + (\partial_iA^{ij})(\partial_j v) \, \mu + A^{ij}(\partial_{ij}v) \, \mu - A^{ij}(\partial_jv)(\partial_iV) \, \mu \Big)
\\
&=\int u \big( F_{\mathrm R} \cdot \nabla v + A:\nabla^2 v \big) \mu
\\
&= (u,L^\star v)_{L^2(\mu)},
\end{align*}
where we have used Einstein's summation notation. 
\end{proof}

The subscript in $F_{\mathrm R}$ is to indicate that this is the drift for the time-reversed diffusion process corresponding to~\eqref{CaseC-eq:SDE} (see \cite[Theorem 2.1]{HaussmannPardoux86} and the proof of Lemma~\ref{CaseC-lem:LZ} below for details). We also need the symmetric part $L_{\sym}$ of the generator $L$:
\begin{align}\label{CaseB-def:SymL}
L_{\sym}h := \frac{1}{2} (L+L^\star)h = F_{\sym}\cdot \nabla h + A:\nabla^2 h, \ \ \text{with} \ \ F_{\sym}=-A\nabla V+\div A.
\end{align}
The operator $L_{\sym}$ is symmetric in $L^2(\mu)$ and satisfies, for any test functions $u$ and $v$, that
\begin{align} \label{eq:chaud5}
(L_{\sym}u,v)_{L^2(\mu)}
&=\int \big( F_{\sym}^i\partial_i u + A^{ij}\partial_{ij} u \big) \, v \, \mu
\nonumber
\\
&= \int \partial_iu \, \Big( (\partial_j A^{ij}) \, v \, \mu - A^{ij} (\partial_jV) \, v \, \mu - (\partial_j A^{ij}) \, v \, \mu - A^{ij} (\partial_jv) \, \mu + A^{ij} (\partial_jV) \, v \, \mu \Big)
\nonumber
\\
&= - \int (\nabla v) \cdot A \nabla u \, \mu.
\end{align}
We next define $\Pi:\R^d\rightarrow \R^{d\times d}$ by
\begin{align*}
\Pi:=\Id_d - \frac{1}{A^{11}} \ e^1 \otimes (Ae^1) = \begin{pmatrix} 0 & -\dfrac{(A^{1,2:d})^T}{A^{11}} \\ \noalign{\vskip 3pt} 0 & \Id_{d-1} \end{pmatrix},
\end{align*}
where $\Id_d$ is the $d$-dimensional identity matrix, $e^1:=(1,0,\ldots,0)^T\in\R^d$ and $A^{1,2:d}:=(A^{21},\ldots,A^{d1})^T \in \R^{d-1}$. Recall that, for any vectors $u$, $v$ and $w$ in $\R^d$, we have $(u \otimes v) w = (v \cdot w) u$. It is easy to verify that $\Pi$ is the orthogonal projection onto the orthogonal of $\mathrm{span}\{e^1 \}$ for the scalar product $(u,v)_A=u^TAv$. In particular, when $A$ is a diagonal matrix, $\Pi$ is the orthogonal projection for the Euclidean scalar product onto $\mathrm{span}\{e^2,\ldots,e^d\}$, where $e^i$ is the $i$-th vector of the canonical basis of $\R^d$. We refer to~\cite[Section 2.1]{LelievreZhang18} for a detailed analysis of this operator. Using the symmetry of $A$, it follows that the first row and the first column of the matrix $A\Pi$ are equal to zero. We can therefore write
\begin{align}\label{def:MatB}
A\Pi= \begin{pmatrix}0 & 0 \\ 0 & B\end{pmatrix} \text{ with } B:=A^{2:d,2:d} - \frac{1}{A^{11}} \, A^{1,2:d} \otimes A^{1,2:d} \in \R^{(d-1) \times (d-1)},
\end{align}
where $A^{2:d,2:d}$ is the submatrix of $A$ without the first row and the first column. Since $A$ is symmetric and positive-definite, we observe that $B$ is also symmetric and positive-definite.

The operator $L_{\sym}$ can also be written as $L_{\sym}h = -A \nabla V \cdot \nabla h + \div (A \nabla h)$. We next define the family $L^{z}_{\sym}$ of operators indexed by $z\in\R$: for any test function $h:\R^d\rightarrow\R$,
\begin{align}\label{CaseB-def:SymLz}
(L^{z}_{\sym} h)(z,x^2_d)
:=
-(A\Pi)(z,x^2_d)\nabla V(z,x^2_d)\cdot \nabla h(z,x^2_d) +\div \big[ A\Pi(z,x^2_d)\nabla h(z,x^2_d) \big].
\end{align}
Note that, in the case $A = \Id_d$, we recover the definition~\eqref{CaseA-def:SymLz}. Using~\eqref{def:MatB}, we see that
$$
(L^{z}_{\sym} h)(z,x^2_d)
=
\sum\limits_{i,j=2}^d \Bigl(-B^{ij}(z,x^2_d) \, \partial_i V(z,x^2_d) \, \partial_j h(z,x^2_d) + \partial_i B^{ij}(z,x^2_d) \, \partial_j h(z,x^2_d) + B^{ij}(z,x^2_d) \, \partial_{ij}h(z,x^2_d) \Bigr).
$$
The operator $L^{z}_{\sym}$ can be interpreted as the projection (with respect to the scalar product induced by the matrix $A$) of the full generator $L_{\sym}$ onto $\R^{d-1}$, for any fixed $z\in\R$. Note that, for any $z\in\R$, the operator $L^{z}_{\sym}$ is self-adjoint in $L^2(\overline{\mu}_{z})$ with, for any test functions $u,v:\R^{d-1}\rightarrow\R$,
\begin{align}\label{LevSet-SelfAdj}
-\int_{\R^{d-1}}(L^z_{\sym} u) \, v \, \overline{\mu}_{z} = \sum_{i,j=2}^d \int_{\R^{d-1}} B^{ij} \, \partial_i u \, \partial_jv \, \overline{\mu}_{z} = \int_{\R^{d-1}} \widehat\nabla v \cdot B \widehat\nabla u \, \overline{\mu}_{z}.
\end{align}
In what follows, we use the notation $\widehat\nabla$ for the gradient on $\R^{d-1}$, i.e. $\widehat\nabla h=(\partial_2 h,\ldots,\partial_d h)^T$ for any $h:\R^d\rightarrow\R$, and the notations
\begin{equation} \label{eq:def-normA}
|v|_A^2 := v^TAv, \quad |w|_B^2 := w^TBw
\end{equation}
for any vector $v\in\R^d$ and any vector $w \in \R^{d-1}$, where we recall that the matrices $A\in\R^{d\times d}$ and $B\in\R^{(d-1)\times (d-1)}$ are defined by~\eqref{eq:defA} and~\eqref{def:MatB}, respectively.

\subsection{Main result}\label{CaseC-sec:MainRes}

We now state the main result comparing the projected dynamics~\eqref{CaseC-eq:Pro} and the effective dynamics~\eqref{CaseC-def:ED1}. 

\begin{thm}\label{CaseC-thm:Strong}
In addition to~\ref{CaseB-ass:Stat}-\ref{CaseA-ass:b-Reg2}, assume that
\begin{enumerate}[topsep=0pt,label=({D}\arabic*)]
\item\label{CaseC-ass:Stat} The system starts at equilibrium, i.e. $X_0\sim\mu$ and $Z_0=X^1_0$.
\item\label{CaseC-ass:Poinc} The family of conditional invariant measures $\overline{\mu}_{x^1}\in \mathcal P(\R^{d-1})$ and the Dirichlet form induced by $L^{x^1}_{\sym}$ satisfy a Poincar{\'e} inequality, uniformly in $x^1\in\R$, with constant $\alpha_{\PI}$: for any $x^1\in\R$ and any $h:\R^{d-1}\rightarrow\R$ for which the right hand side below is finite, we assume that
\begin{align}\label{Poin-Ineq:DirForm}
\int_{\R^{d-1}} \Bigr(h-\int_{\R^{d-1}}h \, \overline{\mu}_{x^1}\Bigl)^2 \, \overline{\mu}_{x^1} \leq \frac{1}{\alpha_{\PI}} \int_{\R^{d-1}} \abs{\widehat\nabla h}^2_{B} \, \overline{\mu}_{x^1},
\end{align}
where we recall that the notation $\abs{\cdot}_{B}$ is defined by~\eqref{eq:def-normA}.
\item\label{CaseC-ass:Coeff} The effective coefficient $b$ (resp. $\sigma$) is Lipschitz with constant $L_{b}$ (resp. $L_{\sigma}$). Furthermore the projected and the effective diffusion coefficients satisfy $\left| \Sigma^{1} \right| \in L^2(\mu)$ and $\sigma \in L^2(\xi_\#\mu)$ respectively, with $\xi(x)=x^1$.
\item\label{CaseC-ass:KaLa} The projected coefficients satisfy $\dps \left| \widehat\nabla F^1 \right|_B \in L^2(\mu)$ and $\dps \left| \widehat\nabla \left| \Sigma^{1} \right| \right|_B \in L^2(\mu)$ with
\begin{align*}
\kappa^2 := \int_{\R^d}\abs{\widehat\nabla F^1}_B^2 \, \mu < \infty, \qquad \lambda^2:=\int_{\R^{d}}\abs{\widehat\nabla \abs{\Sigma^1}}_B^2 \, \mu<\infty.
\end{align*}
\end{enumerate}
Then we have 
\begin{align}\label{CaseC-ED1-Strong}
\E\pra{\sup\limits_{t\in [0,T]}\abs{X^1_t-Z_t}^2}\leq e^{CT}\bra{54T \, \frac{\kappa^2}{\alpha^2_{\PI}}+64 T \, \frac{\lambda^2}{\alpha_{\PI}}},
\end{align}
where $C=\max\{4L_{b},32L^2_{\sigma}\}$.
\end{thm}
The proof of this theorem is postponed until Section~\ref{CaseC-sec:Proof}.

Note that, in contrast to Theorem~\ref{CaseA-thm:Strong}, we have assumed here that $b$ is Lipschitz, and not only one-sided Lipschitz. We comment further on this in Remark~\ref{rem:OneSided_or_Lipschitz} below. Note also that, in the simpler case $A^{1,2:d}=0$, i.e. $\dps A=\left( \begin{matrix} A^{11} & 0 \\ 0& A^{2:d,2:d} \end{matrix}\right)$, we have $B=A^{2:d,2:d}$ (recall the definition~\eqref{def:MatB}). In this setting, the assumption~\eqref{Poin-Ineq:DirForm} follows from the usual Poincar{\'e} inequality~\eqref{PE-ass:NonGrad-PI} if $A$ is uniformly positive definite, i.e. there exists $\vep>0$ such that $A(x)\geq \vep\Id_d$ for every $x\in\R^d$.

As in Section~\ref{A-sec:MainRes}, we first present a weaker error estimate which is easier to prove (see Proposition~\ref{CaseC-prop:Weak} below). To prove this result, we need the following lemma which controls the difference between the projected and the effective coefficients. 

\begin{lem}\label{CaseC-lem:ED1-L2}
Assume that~\ref{CaseC-ass:Poinc} and~\ref{CaseC-ass:KaLa} hold. Then we have
\begin{align*}
\int_{\R^d}\bra{F^1(x)- b(x^1)}^2 \, \mu(dx) \leq \frac{\kappa^2}{\alpha_{\PI}}, \qquad \int_{\R^d} \bra{|\Sigma^{1}|(x)-\sigma(x^1)}^2 \, \mu(dx) \leq 2 \frac{\lambda^2}{\alpha_{\PI}}.
\end{align*} 
\end{lem}

\begin{proof}
The proof of the first inequality follows as in Lemma~\ref{CaseA-lem:fL2}. For the second inequality, we note that
\begin{align*}
\int_{\R^{d-1}}\bra{|\Sigma^{1}|(x^1,x^2_d)-\sigma(x^1)}^2 \, \overline{\mu}_{x^1}(dx^2_d)
&=
2\sigma^2({x^1})-2\sigma({x^1})\int_{\R^{d-1}} |\Sigma^{1}|({x^1},x^2_d) \, \overline{\mu}_{x^1}(dx^2_d)
\\
&\leq
2\sigma^2({x^1})-2\pra{\int_{\R^{d-1}} |\Sigma^{1}|({x^1},x^2_d) \, \overline{\mu}_{x^1}(dx^2_d)}^2
\\
&=
2 \int_{\R^{d-1}}\pra{|\Sigma^{1}|({x^1},x^2_d) - \int_{\R^{d-1}} |\Sigma^{1}|({x^1},x^2_d) \, \overline{\mu}_{x^1}(dx^2_d)}^2 \, \overline{\mu}_{x^1}(dx^2_d)
\\
&\leq
\frac{2}{\alpha_{\PI}} \int_{\R^{d-1}} \abs{\widehat\nabla|\Sigma^{1}|({x^1},x^2_d)}_B^2 \, \overline{\mu}_{x^1}(dx^2_d).
\end{align*}
Here the first inequality follows from the Cauchy-Schwarz inequality and the second inequality follows from Assumption~\ref{CaseC-ass:Poinc}. The result then follows by using the disintegration theorem and Assumption~\ref{CaseC-ass:KaLa}.
\end{proof}

\begin{prop}[Weak error estimate]\label{CaseC-prop:Weak}
In addition to~\ref{CaseB-ass:Stat}-\ref{CaseA-ass:b-Reg2}, assume that~\ref{CaseC-ass:Stat}-\ref{CaseC-ass:KaLa} hold. Then we find
\begin{align}\label{CaseC-WeakErr}
\E\pra{\sup\limits_{t\in [0,T]}\abs{X^1_t-Z_t}^2} \leq e^{CT}\bra{4T^2 \, \frac{\kappa^2}{\alpha_{\PI}} + 64 T \, \frac{\lambda^2}{\alpha_{\PI}}},
\end{align}
where $C=\max\{4L_{b},32L^2_{\sigma}\}$. 
\end{prop}

\begin{proof}
Using the projected dynamics~\eqref{CaseC-eq:Pro} and the effective dynamics~\eqref{CaseC-def:ED1}, we write that
\begin{multline*}
X^1_s - Z_s
=
\int_0^s \bra{F^1(X_r)- b(X^1_r)} dr + \int_0^s \bra{ b(X^1_r)-b(Z_r)}dr
\\
+
\sqrt{2} \int_0^s \bra{|\Sigma^1|(X_r)- \sigma(X^1_r)} dB_r + \sqrt{2} \int_0^s \bra{ \sigma(X^1_r)-\sigma(Z_r)}dB_r.
\end{multline*}
Introducing the functions $f,g:\R^d\rightarrow\R$ defined by $f(x)=F^1(x)- b(x^1)$ and $g(x)=|\Sigma^{1}|(x)-\sigma(x^1)$, and using Young's inequality, we obtain that, for any $t\in [0,T]$,
\begin{multline}\label{CaseC-Gron}
\E\pra{\sup\limits_{s\in [0,t]}\abs{X^1_s-Z_s}^2} \leq 4\E\pra{\sup\limits_{s\in [0,t]}\abs{\int_0^s f(X_r) \, dr}^2}+4\E\pra{\sup\limits_{s\in [0,t]}\abs{\int_0^s \bra{ b(X^1_r)- b(Z_r)} \, dr}^2}
\\
+8\E\pra{\sup\limits_{s\in [0,t]}\abs{\int_0^s g(X_r) \, dB_r}^2} +8\E\pra{\sup\limits_{s\in [0,t]}\abs{\int_0^s \bra{ \sigma(X^1_r)- \sigma(Z_r)} \, dB_r}^2}.
\end{multline}
We successively estimate the four terms of the right hand side of~\eqref{CaseC-Gron}. The last term is estimated by using the Doob's inequality followed by It{\^o} isometry:
\begin{align}
\E\pra{\sup\limits_{s\in [0,t]}\abs{\int_0^s \bra{ \sigma(X^1_r)- \sigma(Z_r)} dB_r}^2}
&\leq
4\E\pra{\abs{\int_0^t \bra{ \sigma(X^1_r)- \sigma(Z_r)}dB_r}^2}
\nonumber
\\
&=
4\E\pra{\int_0^t\abs{ \sigma(X^1_r)- \sigma(Z_r)}^2dr}
\nonumber
\\
&\leq
4L^2_{\sigma} \, \E\pra{\int_0^t\abs{ X^1_r-Z_r}^2dr}
\nonumber
\\
&\leq
4L^2_{\sigma} \, \E\pra{\int_0^t \sup\limits_{\tau\in [0,r]}\abs{ X^1_\tau-Z_\tau}^2dr}.
\label{eq:chaud6}
\end{align}
We now estimate the second term of~\eqref{CaseC-Gron} as follows:
\begin{multline} \label{eq:chaud7}
\E\pra{\sup\limits_{s\in [0,t]}\abs{\int_0^s \bra{ b(X^1_r)- b(Z_r)}dr}^2}
\leq
\E\pra{\sup\limits_{s\in [0,t]}\bra{\int_0^s \abs{ b(X^1_r)- b( Z_r)}dr}^2}
\\
\leq
L^2_{b} \, \E\pra{\bra{\int_0^t \sup\limits_{\tau\in[0,r]}\abs{X^1_\tau- Z_\tau}dr}^2}.
\end{multline}
The first and third terms in~\eqref{CaseC-Gron} are estimated by using Lemma~\ref{CaseC-lem:ED1-L2} along with the stationarity assumption $X_0\sim\mu$, which gives, using Cauchy Schwarz and Doob's inequalities, that
\begin{multline}
\E\pra{\sup\limits_{s\in [0,t]}\abs{\int_0^s f(X_r) \, dr}^2}
\leq
\E\pra{\sup\limits_{s\in [0,T]}\abs{\int_0^s f(X_r) \, dr}^2}
\\
\leq
T \, \E\pra{\int_0^T\abs{f(X_r)}^2dr}
=
T^2 \int_{\R^d} f^2 \, \mu
\leq
T^2 \, \frac{\kappa^2}{\alpha_{\PI}}, \label{CaseC-eq:Weak-f}
\end{multline}
and
\begin{multline}
\E\pra{\sup\limits_{s\in [0,t]}\abs{\int_0^s g(X_r) \, dB_r}^2}
\leq
4\E\pra{\abs{\int_0^T g(X_r) \, dB_r}^2}
\\
=
4\E\pra{\int_0^T\abs{g(X_r)}^2dr}
=
4T\int_{\R^d} g^2 \, \mu
\leq
8T \, \frac{\lambda^2}{\alpha_{\PI}}. \label{CaseC-eq:Weak-g}
\end{multline}
Collecting~\eqref{CaseC-Gron}, \eqref{eq:chaud6}, \eqref{eq:chaud7}, \eqref{CaseC-eq:Weak-f} and~\eqref{CaseC-eq:Weak-g}, we obtain that
\begin{align}\label{CaseC-Gron-2}
\E\pra{\sup\limits_{s\in [0,t]}\abs{X^1_s-Z_s}^2} \leq C_{f,g} + \phi(t),
\end{align}
where $\dps C_{f,g} := 4T^2 \, \frac{\kappa^2}{\alpha_{\PI}} + 64 T \, \frac{\lambda^2}{\alpha_{\PI}}$ and 
\begin{align*}
\phi(t):=4L^2_{b} \, \E\pra{\bra{\int_0^t \sup\limits_{\tau\in[0,r]}\abs{X^1_\tau- Z_\tau}dr}^2}+ 32L^2_{\sigma} \, \E\pra{\int_0^t \sup\limits_{\tau\in [0,r]}\abs{ X^1_\tau-Z_\tau}^2dr}.
\end{align*}
Since $X^1_t$ and $Z_t$ are continuous-time stochastic processes, we can interchange expectation and time-derivative and compute that
\begin{multline*}
\frac{d\phi}{dt}
=
8L^2_{b} \, \E\left[\left( \sup\limits_{s\in[0,t]} \abs{X^1_s- Z_s} \right) \left(\int_0^t \sup\limits_{\tau\in[0,r]}\abs{X^1_\tau- Z_\tau}dr \right) \right]
+
32L^2_{\sigma} \, \E\left[ \sup\limits_{s\in [0,t]}\abs{ X^1_s-Z_s}^2 \right]
\\
\leq
8L^2_{b} \, \left[ \E\left[ \sup\limits_{s\in[0,t]}\abs{X^1_s- Z_s}^2 \right] \right]^{1/2} \left[ \E\left[ \left( \int_0^t \sup\limits_{\tau\in[0,r]}\abs{X^1_\tau- Z_\tau}dr\right)^2\right] \right]^{1/2}
+
32L^2_{\sigma} \, \E\pra{ \sup\limits_{s\in [0,t]}\abs{ X^1_s-Z_s}^2}.
\end{multline*}
Using~\eqref{CaseC-Gron-2}, we deduce that
$$
\frac{d\phi}{dt}
\leq
8L^2_{b} \, \sqrt{\frac{\phi(t)}{4L^2_{b}}} \ \sqrt{C_{f,g}+\phi(t)} + 32 L^2_{\sigma} \, \big( C_{f,g}+\phi(t) \big) \leq \widetilde{C} \bra{C_{f,g}+\phi(t)},
$$
where $\widetilde{C} :=\max\{4L_{b}, 32L^2_{\sigma}\}$. Using the Gronwall lemma, we find that, for any $t\in [0,T]$,
\begin{align*}
\phi(t) \leq \left( e^{\widetilde{C} t}-1 \right) C_{f,g} \leq \left( e^{\widetilde{C} T}-1 \right) C_{f,g}.
\end{align*} 
Substituting into~\eqref{CaseC-Gron-2}, we get
\begin{align*}
\E\pra{\sup\limits_{s\in [0,t]}\abs{X^1_s-Z_s}^2} \leq e^{\widetilde{C}T}C_{f,g},
\end{align*}
which is the claimed estimate~\eqref{CaseC-WeakErr}. This concludes the proof of Proposition~\ref{CaseC-prop:Weak}.
\end{proof}

\subsection{Proof of Theorem~\ref{CaseC-thm:Strong}}\label{CaseC-sec:Proof}

In this section, we prove Theorem~\ref{CaseC-thm:Strong}. The main challenge (as in Section~\ref{A-sec:Proof}) is to prove a sharper estimate on
\begin{align*}
\E\pra{\sup\limits_{t\in[0,T]}\abs{\int_0^tf(X_s) \, ds}^2},
\end{align*}
which, in the weak estimate of Proposition~\ref{CaseC-prop:Weak}, is controlled by $\|f\|_{L^2(\mu)}$ (see~\eqref{CaseC-eq:Weak-f}). The strategy of the proof mirrors that used for the case of identity diffusion matrix (see the beginning of Section~\ref{A-sec:Proof}): (1)~use an argument due to Lyons-Zhang to estimate expectations of the type $\dps \E \left[ \sup\limits_{t \in[0,T]} \abs{\int_0^t \nabla^\star (A\Phi(X_s)) \, ds}^2 \right]$, (2)~make an appropriate choice for $\Phi$ such that $\nabla^\star (A\Phi)=f$, and (3)~complete the proof with a Gronwall argument. 

We point out that, ideally, one would like to also strengthen in a similar way the estimate on $\dps \E \left[ \sup\limits_{t\in[0,T]}\abs{\int_0^tg(X_s) \, ds}^2 \right]$ which appears in the proof, see~\eqref{CaseC-eq:Weak-g}. However, repeating for this term the same arguments as the ones described above seems challenging (see Remark~\ref{rem:DubSch} for more details).

In the following two lemmas, we focus on the first two steps described above. In Lemma~\ref{CaseC-lem:LZ}, we apply the Lyons-Zhang argument to our case and in Lemma~\ref{CaseC-lem:Pois-Lev-Set}, we solve an auxiliary Poisson problem on the level sets required for the final result. We need the function spaces
\begin{align*}
L^2(\mu,A):= \Bigl\{ h:\R^d\rightarrow\R^d, \ \ \int_{\R^d} |h|^2_A \, \mu <\infty \Bigr\}, \qquad H^{1}(\mu,A):= \Bigl\{ h\in L^2(\mu), \ \ \int_{\R^d} |\nabla h|^2_{A} \, \mu<\infty \Bigr\},
\end{align*}
endowed with the respective norms
\begin{align*}
\| h\|^2_{L^2(\mu,A)} := \| \, |h|_A \, \|^2_{L^2(\mu)}, \qquad \|h\|^2_{H^1(\mu,A)} := \|h\|^2_{L^2(\mu)} + \| \, |\nabla h|_A \, \|^2_{L^2(\mu)},
\end{align*}
where we recall that the notation $|\cdot|_A$ is defined by~\eqref{eq:def-normA}.

\begin{lem}\label{CaseC-lem:LZ}
Let $(X_t)_{t\geq 0}$ be the solution to~\eqref{CaseC-eq:SDE} with initial condition distributed according to the equilibrium measure $\mu$ (see Assumption~\ref{CaseC-ass:Stat}). Consider a function $\Phi:\R^d\rightarrow\R^d$ such that $\Phi\in \left( C^\infty \right)^d \cap L^2(\mu,A)$. Then, for any $T>0$, we have 
\begin{align}\label{CaseB-eq:NonGrad-LyonZhang-Res}
\E\Biggl[ \sup\limits_{t\in[0,T]} \Biggl| \int_0^t \nabla^\star (A\Phi(X_s)) \, ds \Biggr|^2 \Biggr] \leq \frac{27}{2}T \, \| \, |\Phi|_A \, \|^2_{L^2(\mu)}. 
\end{align}
Here $\nabla^\star$ is the adjoint of the operator $\nabla$ with respect to the $L^2(\mu)$ inner product, and satisfies 
\begin{align*}
\nabla^\star (A\Phi) = \nabla V \cdot A\Phi-\div(A\Phi) = \sum_{i=1}^d \partial_i V \, (A\Phi)^{i} -\partial_i (A\Phi)^{i}.
\end{align*}
\end{lem}

\begin{proof}
The proof falls in two steps. 

\emph{Step 1.} For any $\eta>0$, consider the resolvent problem
\begin{align}\label{PE-eq:NonGrad-Mob-Res-Pb}
\eta w_\eta-L_{\sym}w_\eta=\nabla^\star (A\Phi),
\end{align}
where $L_{\sym}$, defined by~\eqref{CaseB-def:SymL}, is the symmetric part of $L$. Using~\eqref{eq:chaud5}, we see that the corresponding variational problem is to find $w_\eta\in H^1(\mu,A)$ which solves 
\begin{align*}
\forall v \in H^1(\mu,A), \quad \eta\int_{\R^d}w_\eta\,v\,\mu +\int_{\R^d} \nabla w_\eta\cdot A\nabla v\,\mu =-\int_{\R^d} \Phi\cdot A\nabla v\,\mu.
\end{align*}
Using the Lax-Milgram theorem, this variational problem has a unique solution $w_\eta\in H^1(\mu,A)$ for any $\eta>0$. Choosing $v=w_\eta$, we obtain
\begin{align*}
\eta \|w_\eta\|^2_{L^2(\mu)} + \| \, |\nabla w_\eta|_A \, \|^2_{L^2(\mu)} \leq \| \, |\Phi|_A \, \|_{L^2(\mu)} \| \, |\nabla w_\eta|_A \, \|_{L^2(\mu)}.
\end{align*}
We then deduce that
\begin{align}\label{PE-eq:NonGrad-Mob-Res-Est}
\forall \eta>0, \quad \| \, |\nabla w_\eta|_A \, \|_{L^2(\mu)} \leq \| \, |\Phi|_A \, \|_{L^2(\mu)}
\end{align}
and $\dps \sqrt{\eta} \, \|w_\eta\|_{L^2(\mu)} \leq \| \, |\Phi|_A \, \|_{L^2(\mu)}$, which implies that
\begin{align}\label{PE-eq:NonGrad-Mob-L2Limit}
\lim\limits_{\eta\rightarrow 0} \eta \, \|w_\eta\|_{L^2(\mu)} =0.
\end{align}
Since $A$, $V$, $\mu$ and $\Phi$ are smooth and $A$ is positive definite, it follows that $w_\eta\in C^2$ by standard results from elliptic theory.

\emph{Step 2.} Since $w_\eta$ is smooth, using It{\^o}'s lemma, we write that, for any $t\in [0,T]$,
\begin{align}\label{PE-eq:NonGrad-Mob-Ito}
w_\eta(X_t)-w_\eta(X_0) =\int_0^tLw_\eta(X_s) \, ds + \sqrt{2} \int_0^t \nabla w_\eta(X_s)\cdot \Sigma(X_s) \, dW_s.
\end{align}
For $s\in [0,T]$, we introduce the time-reversed process
\begin{align*}
Y_s=X_{T-s}.
\end{align*}
Since the coefficients of the original dynamics~\eqref{CaseC-eq:SDE} are Lipschitz (see Assumption~\ref{CaseB-ass:MainAss}) and $\mathrm{law}(X_t)=\mathrm{law}(X_0)=\mu$, the time-reversed process solves (see~\cite[Theorem 2.1]{HaussmannPardoux86})
\begin{align*}
dY_s=F_{\mathrm R}(Y_s) \, ds + \sqrt{2} \, \Sigma(Y_s) \, d\overline W_s,
\end{align*}
where $(\overline W_s)_{0\leq s\leq T}$ is a Brownian motion and where we recall that $F_{\mathrm R}:=-A\nabla V+\div A-c$. Note that the generator of the time-reversed process is the adjoint of $L$ in $L^2(\mu)$ (see Lemma~\ref{CaseB-lem:Adj}). Applying It{\^o}'s lemma once more, we get that, for any $t\in [0,T]$,
\begin{align}\label{PE-eq:NonGrad-Mob-Time-Rev-Ito}
w_\eta(Y_T)-w_\eta(Y_{T-t})=\int_{T-t}^T L^\star w_\eta(Y_s) \, ds + \sqrt{2} \, \int_{T-t}^T \nabla w_\eta(Y_s)\cdot \Sigma(Y_s) \, d\overline W_s.
\end{align}
Setting 
\begin{align*}
M_t=\int_0^t\nabla w_\eta(X_s)\cdot \Sigma(X_s) \, dW_s \ \ \text{ and } \ \ \overline M_t=\int_{0}^t\nabla w_\eta(Y_s)\cdot \Sigma(Y_s) \, d\overline W_s,
\end{align*}
adding~\eqref{PE-eq:NonGrad-Mob-Ito} and~\eqref{PE-eq:NonGrad-Mob-Time-Rev-Ito} and using the definition~\eqref{CaseB-def:SymL} of $L_{\sym}$ we find  
\begin{align}
0&= \int_0^t Lw_\eta(X_s) \, ds + \int_{T-t}^T L^\star w_\eta(Y_s) \, ds + \sqrt{2} \, (M_t+\overline M_T-\overline M_{T-t}) \nonumber\\
&= \int_0^t Lw_\eta(X_s) \, ds +\int_{0}^t L^\star w_\eta(Y_{T-s}) \, ds + \sqrt{2} \, (M_t+\overline M_T-\overline M_{T-t}) \nonumber\\
&= \int_0^t Lw_\eta(X_s) \, ds +\int_{0}^t L^\star w_\eta(X_s) \, ds + \sqrt{2} \, (M_t+\overline M_T-\overline M_{T-t}) \nonumber\\
&= 2\int_0^t L_{\sym}w_\eta(X_s) \, ds + \sqrt{2} \, (M_t+\overline M_T-\overline M_{T-t}). \label{PE-eq:NonGrad-Mob-LyonZhang}
\end{align}
The random process $(M_t)_{0\leq t\leq T}$ is a square-integrable martingale since 
\begin{align*}
\E\pra{\int_0^T|\nabla w_\eta \Sigma|^2(X_s) \, ds} = \int_0^T \int_{\R^d} |\nabla w_\eta \Sigma|^2(x) \, \mu(dx) \, ds = T \, \| \, |\nabla w_\eta|_A \, \|^2_{L^2(\mu)} \leq T \, \| \, |\Phi|_A \, \|^2_{L^2(\mu)}<\infty,
\end{align*}
where we use $X_s\sim\mu$ to arrive at the first equality and~\eqref{PE-eq:NonGrad-Mob-Res-Est} to arrive at the first inequality. Using Doob's inequality followed by It{\^o}'s isometry, we deduce that
\begin{align*}
\E\pra{ \bra{ \sup\limits_{t\in [0,T]} |M_t|}^2} \leq 4 \E \pra{  \int_0^T|\nabla w_\eta \Sigma|^2(X_s) \, ds } \leq 4T \, \| \, |\Phi|_A \, \|^2_{L^2(\mu)}.
\end{align*}
Since $\mathrm{law}(X_{T-s})=\mathrm{law}(Y_s)=\mu$, a similar estimate holds for $(\overline M_{t})_{0\leq t\leq T}$. 

Applying the Young's inequality to~\eqref{PE-eq:NonGrad-Mob-LyonZhang} and inserting the above estimate, we find
\begin{multline} \label{eq:chaud8}
\E\pra{ \sup\limits_{t\in [0,T]}\abs{ \int_0^t L_{\sym}w_\eta(X_s) \, ds}^2}
\leq
\frac{3}{2}\bra{ \E\pra{ \sup\limits_{t\in [0,T]}|M_t|^2} + \E\left[|\overline M_T|^2\right] + \E\pra{ \sup\limits_{t\in [0,T]}|\overline M_t|^2}}
\\
\leq
\frac{27}{2} T \, \| \, |\Phi|_A \, \|^2_{L^2(\mu)}.
\end{multline}
We are now in position to bound the left hand side of~\eqref{CaseB-eq:NonGrad-LyonZhang-Res}. Using~\eqref{PE-eq:NonGrad-Mob-Res-Pb}, we have, for any $\nu>0$,
\begin{align*}
\abs{\int_0^t\nabla^\star (A\Phi(X_s)) \, ds }^2 \leq (1+\nu) \abs{\int_0^tL_{\sym}w_\eta(X_s) \, ds }^2 + \bra{ 1+\frac{1}{\nu}} \abs{\int_0^t\eta w_\eta(X_s) \, ds }^2, 
\end{align*}
and therefore, using~\eqref{eq:chaud8},
\begin{align*}
\E\pra{ \sup\limits_{t\in [0,T]} \abs{\int_0^t\nabla^\star (A\Phi(X_s)) \, ds }^2}
&\leq
(1+\nu) \, \frac{27}{2} T \, \| \, |\Phi|_A \, \|^2_{L^2(\mu)} +\bra{ 1+\frac{1}{\nu}} \E\pra{ \sup\limits_{t\in [0,T]}\abs{\int_0^t\eta w_\eta(X_s) \, ds }^2}
\\
&\leq
(1+\nu) \, \frac{27}{2} T \, \| \, |\Phi|_A \, \|^2_{L^2(\mu)} + \bra{ 1+\frac{1}{\nu}} \eta^2T \, \E\pra{ \int_0^T w_\eta^2(X_s) \, ds}
\\
& \leq
(1+\nu) \, \frac{27}{2} T \, \| \, |\Phi|_A \, \|^2_{L^2(\mu)} + \bra{ 1+\frac{1}{\nu}} \eta^2T^2 \, \|w_\eta\|^2_{L^2(\mu)},
\end{align*}
where the second inequality follows from the Cauchy-Schwarz inequality. The result then follows by passing to the limit $\eta\rightarrow 0$, using~\eqref{PE-eq:NonGrad-Mob-L2Limit} and next passing to the limit $\nu\rightarrow 0$. This concludes the proof of Lemma~\ref{CaseC-lem:LZ}.
\end{proof}

We need the following function space for the next result:
\begin{align*}
H^{1}_m(\overline{\mu}_z,B) := \Bigl\{ v\in L^2(\overline{\mu}_z), \quad \int_{\Sigma_z} \left| \widehat\nabla v \right|^2_{B} \, \overline{\mu}_z<\infty \ \text{ and } \ \int_{\Sigma_z}v \, \overline{\mu}_z=0  \Bigr\},
\end{align*}
endowed with the norm
\begin{align*}
\|u\|^2_{H^1(\overline{\mu}_z,B)} := \|u\|^2_{L^2(\overline{\mu}_z)} + \left\| \, \left| \widehat\nabla u \right|_{B} \, \right\|^2_{L^2(\overline{\mu}_z)},
\end{align*}
where we recall that the matrix $B\in\R^{(d-1) \times (d-1)}$, defined by~\eqref{def:MatB}, is the submatrix of $A\Pi$ without the first row and the first column.

\begin{lem}[Poisson problem on the level sets]\label{CaseC-lem:Pois-Lev-Set}
Assume that~\ref{CaseC-ass:Poinc} and~\ref{CaseC-ass:KaLa} hold. Consider some function $\ell \in C^\infty(\R^d;\R)$ with $\ell(z,\cdot) \in L^2(\overline{\mu}_z)$ and $\dps \int_{\R^{d-1}} \ell(z,x^2_d) \, \overline{\mu}_z(dx^2_d)=0$ for any $z\in\R$. Then, for any $z\in \R$, there exists a unique solution $x^2_d\mapsto u(z,x^2_d)\in H^1_m(\overline{\mu}_z,B)$ to 
\begin{align}\label{PE-eq:C-NonGrad-SloMob-Pois-LevSet}
-(L^z_{\sym})u=\ell(z,\cdot) \quad \text{with} \quad \int_{\R^{d-1}}u(z,x^2_d) \, \overline{\mu}_z(dx^2_d)=0,
\end{align}
where we recall that $L^z_{\sym}$ is defined by~\eqref{CaseB-def:SymLz}. Furthermore, $u$ is a smooth function which satisfies
\begin{align}\label{PE-eq:NonGrad-SloMob-LevelSet-Grad}
\forall z \in \R, \quad \left\| \, \left| \widehat\nabla u(z,\cdot) \right|_{B} \, \right\|^2_{L^2(\overline{\mu}_z)} \leq \frac{1}{\alpha_{\PI}} \, \|\ell(z,\cdot)\|^2_{L^2(\overline{\mu}_z)}.
\end{align}
\end{lem}

\begin{proof}
Using~\eqref{LevSet-SelfAdj}, the variational problem corresponding to~\eqref{PE-eq:C-NonGrad-SloMob-Pois-LevSet} is to find $u(z,\cdot)\in H^1_m(\overline{\mu}_z,B)$ which solves
\begin{align}\label{PE-eq:NonGrad-SloMob-LevSet-VarForm}
\forall v\in H^1_m(\overline{\mu}_z,B), \quad \int_{\R^{d-1}} \widehat \nabla u(z,\cdot) \cdot B \, \widehat \nabla v \,\overline{\mu}_z =\int_{\R^{d-1}} \ell(z,\cdot) \, v\,\overline{\mu}_z.
\end{align}
The coercivity of the bounded bilinear form on the left hand side follows from Assumption~\ref{CaseC-ass:Poinc}. For any $z\in \R$, the above right hand side is well-defined since $v\in H^1_m(\overline{\mu}_z,B)$ and $\ell\in L^2(\overline{\mu}_z)$. Using the Lax-Milgram theorem, the variational problem~\eqref{PE-eq:NonGrad-SloMob-LevSet-VarForm} therefore admits a unique solution.

Since $\mu$, $\ell$, $V$ and $B$ are smooth, it follows that $u(z,\cdot)$ is smooth for any $z\in\R$. Choosing $v=u(z,\cdot)$ in~\eqref{PE-eq:NonGrad-SloMob-LevSet-VarForm} and using~\eqref{Poin-Ineq:DirForm}, we have
\begin{multline*}
\int_{\R^{d-1}} \left| \widehat\nabla u(z,\cdot) \right|^2_{B} \, \overline{\mu}_z
=
\int_{\R^{d-1}} \ell(z,\cdot) \, u(z,\cdot) \, \overline{\mu}_z
\\
\leq
\|\ell(z,\cdot)\|_{L^2(\overline{\mu}_z)} \, \|u(z,\cdot)\|_{L^2(\overline{\mu}_z)}
\leq
\frac{1}{\sqrt{\alpha_{\PI}}} \|\ell(z,\cdot)\|_{L^2(\overline{\mu}_z)} \left\| \, \left| \widehat\nabla u(z,\cdot) \right|_B \, \right\|_{L^2(\overline{\mu}_z)}, 
\end{multline*}
and therefore~\eqref{PE-eq:NonGrad-SloMob-LevelSet-Grad} follows.
\end{proof}

We now combine Lemmas~\ref{CaseC-lem:LZ} and~\ref{CaseC-lem:Pois-Lev-Set} to prove Theorem~\ref{CaseC-thm:Strong}.

\begin{proof}[Proof of Theorem~\ref{CaseC-thm:Strong}]
As in the proof of Proposition~\ref{CaseC-prop:Weak} (see~\eqref{CaseC-Gron}), we have, for any $t\in [0,T]$, that
\begin{multline}\label{CaseC-GronGen}
\E\pra{\sup\limits_{s\in [0,t]}\abs{X^1_s-Z_s}^2} \leq 4\E\pra{\sup\limits_{s\in [0,t]}\abs{\int_0^s f(X_r) \, dr}^2}+4\E\pra{\sup\limits_{s\in [0,t]}\abs{\int_0^s \bra{ b(X^1_r)- b(Z_r)}dr}^2}
\\
+8\E\pra{\sup\limits_{s\in [0,t]}\abs{\int_0^s g(X_r) \, dB_r}^2} +8\E\pra{\sup\limits_{s\in [0,t]}\abs{\int_0^s \bra{ \sigma(X^1_r)- \sigma(Z_r)}dB_r}^2},
\end{multline}
where we recall that $f,g:\R^d\rightarrow\R$ are defined by $f(x)=F^1(x)-b(x^1)$ and $g(x)=|\Sigma^{1}|(x)-\sigma(x^1)$.

The last three terms of the right hand side of~\eqref{CaseC-GronGen} are estimated as in the proof of Proposition~\ref{CaseC-prop:Weak}, see~\eqref{eq:chaud6}, \eqref{eq:chaud7} and~\eqref{CaseC-eq:Weak-g}. We now estimate the first term on the right hand side of~\eqref{CaseC-GronGen} in a sharper way than in the proof of Proposition~\ref{CaseC-prop:Weak}. Since $\dps \int_{\R^{d-1}}f(z,\cdot) \, \overline{\mu}_z=0$ and $f\in L^2(\overline{\mu}_z)$ (see Lemma~\ref{CaseC-lem:ED1-L2}), we can choose $\ell=f$ in Lemma~\ref{CaseC-lem:Pois-Lev-Set}, which implies that there exists a unique solution to 
\begin{align} \label{eq:chaud9}
-(L^{x^1}_{\sym})u=f(x^1,\cdot) \ \ \text{ with } \ \ \int_{\R^{d-1}}u(x^1,x^2_d) \, \overline{\mu}_{x^1}(dx^2_d)=0,
\end{align}
in the sense of Lemma~\ref{CaseC-lem:Pois-Lev-Set}. Additionally, integrating~\eqref{PE-eq:NonGrad-SloMob-LevelSet-Grad} with respect to $\widehat\mu(dz)$ and using Lemma~\ref{CaseC-lem:ED1-L2}, we find that
\begin{align} \label{eq:chaud10}
\left\| \, \left| \widehat\nabla u \right|_{B} \, \right\|^2_{L^2(\mu)} \leq \frac{1}{\alpha_{\PI}} \| f \|_{L^2(\mu)}^2 \leq \frac{\kappa^2}{\alpha_{\PI}^2}.
\end{align}
Next, we make the choice $\Phi=\Pi\nabla u:\R^d\rightarrow\R^d$ in Lemma~\ref{CaseC-lem:LZ}, which is a valid choice since $u$ and $\Pi$ are smooth and
$\dps \left\| \, \left| \Pi\nabla u \right|_{A} \, \right\|^2_{L^2(\mu)} = \left\| \, \left| \widehat\nabla u \right|_{B} \, \right\|^2_{L^2(\mu)}<\infty$. In this case, using~\eqref{CaseB-def:SymLz} and~\eqref{eq:chaud9}, we compute that
\begin{align*}
-\nabla^\star (A\Phi)
=-\nabla V\cdot(A\Pi\nabla u) + \div(A\Pi\nabla u) 
=
L^{x^1}_{\sym}u
=
-f(x^1,\cdot).
\end{align*}
Therefore, with this choice in~\eqref{CaseB-eq:NonGrad-LyonZhang-Res} and using~\eqref{eq:chaud10}, we find
\begin{align*}
\E\Biggl[ \sup\limits_{t\in[0,T]}\Biggl| \int_0^t f(X_s) \, ds\Biggr|^2\Biggr]
\leq
\frac{27}{2}T \, \| \, \left| \Pi\nabla u \right|_{A} \, \|^2_{L^2(\mu)}
=
\frac{27}{2}T \, \left\| \, \left| \widehat\nabla u \right|_{B} \, \right\|^2_{L^2(\mu)}
\leq
\frac{27}{2}T \, \frac{\kappa^2}{\alpha_{\PI}^2}. 
\end{align*}
Collecting this estimate with~\eqref{CaseC-GronGen}, \eqref{eq:chaud6}, \eqref{eq:chaud7} and~\eqref{CaseC-eq:Weak-g}, and repeating the Gronwall-type argument of the proof of Proposition~\ref{CaseC-prop:Weak}, we obtain the estimate~\eqref{CaseC-ED1-Strong}. This concludes the proof of Theorem~\ref{CaseC-thm:Strong}.
\end{proof}

\begin{rem} \label{rem:OneSided_or_Lipschitz}
Note that, in contrast to the case $\Sigma=\Id_d$ (considered in Theorem~\ref{CaseA-thm:Strong}) where the effective drift $b$ is one-sided Lipschitz, we assume here that $b$ is Lipschitz. This is because we do not know how to obtain an appropriate error bound in the case $\Sigma \neq \Id_d$ if $b$ is only one-sided Lipschitz. Indeed, if we try to control $(X^1_s-Z_s)^2$ via the It{\^o}'s formula, it is unclear to us how to get an upper bound of the order of $1/\alpha^2_{\PI}$ on the term $\dps \E\pra{\sup\limits_{s\in[0,T]} \abs{\int_0^s \big(X^1_r-Z_r\big) \, \big( F^1(X_r)-b(Z_r) \big) \, dr}}$. Alternatively, if we try to control $|X^1_s-Z_s|$ as in the proof of Theorem~\ref{CaseA-thm:Strong}, getting an appropriate upper bound for the stochastic-integral term seems challenging.
\end{rem}

\begin{rem}\label{rem:PepinSetup}
Both the weak and the strong error estimates can be strengthened in the case when the projected diffusion coefficient only depends on the projected variable, i.e. $|\Sigma^1|(x)=|\Sigma^1|(x^1)$. In this case, we have $\sigma(x^1)=|\Sigma^{1}|(x^1)$. This particular case often appears in the averaging literature (see for instance~\cite{Pepin17}). Sharper estimates arise because both the projected and the effective dynamics have the same diffusion coefficient:
\begin{align*}
dX^1_t=F^1(X_t) \, dt + \sqrt{2} \, |\Sigma^{1}|(X^1_t) \, dB_t, \qquad dZ_t= b(Z_t) \, dt +\sqrt{2} \, |\Sigma^{1}|(Z_t) \, dB_t,
\end{align*}
where $B_t$ is a one-dimensional Brownian motion. Since $\widehat\nabla|\Sigma^1|=0$, we have $\lambda=0$ in Assumption~\ref{CaseC-ass:KaLa}. Under the same assumptions as in Theorem~\ref{CaseC-thm:Strong}, the estimates~\eqref{CaseC-WeakErr} and~\eqref{CaseC-ED1-Strong} respectively become
\begin{align*}
\text{(weak error estimate)} \qquad & \E\pra{\sup\limits_{t\in [0,T]}\abs{X^1_t-Z_t}^2}\leq 4 \, e^{C T} \, T^2 \, \frac{\kappa^2}{\alpha_{\PI}},\\
\text{(strong error estimate)}\qquad & \E\pra{\sup\limits_{t\in [0,T]}\abs{X^1_t-Z_t}^2}\leq 54 \, e^{C T} \, T \, \frac{\kappa^2}{\alpha_{\PI}^2},
\end{align*}
where $C=\max\{4L_{b}, 32 L^2_{\sigma}\}$. Since the projected and the effective dynamics have the same diffusion coefficient, only its Lipschitz constant $L_{\sigma}$ arises when analysing the difference $X^1_t-Z_t$.
\end{rem}

\begin{rem}\label{rem:DubSch}
Note that the stronger error estimate~\eqref{CaseC-ED1-Strong} has the same scaling in terms of the Poincar{\'e} constant $\alpha_{\PI}$ in the diffusion part (characterized by the $\lambda$-term), as compared to the weaker estimate~\eqref{CaseC-WeakErr}. This is due to the fact that we cannot estimate the term $\dps \E \left[ \sup\limits_{t\in[0,T]}\abs{\int_0^t g(X_s) \, dB_s}^2 \right]$ (which appears in the error estimate, see~\eqref{CaseC-Gron}), where $g(x)=|\Sigma^1|(x)-\sigma(x^1)$, using the arguments of Lemmas~\ref{CaseC-lem:LZ} and~\ref{CaseC-lem:Pois-Lev-Set}.
\end{rem}

\begin{rem}
\label{rem:random_clock}
We now present an alternative way to construct the effective dynamics. The main issue in estimating the term $\dps \E \left[ \sup\limits_{t\in[0,T]}\abs{\int_0^t g(X_s) \, dB_s}^2 \right]$ is the presence of the Brownian motion in the integral. To deal with this, an alternative strategy consists in using the Dubins-Schwarz theorem and reformulating the projected dynamics~\eqref{CaseC-eq:Pro} as
\begin{align} \label{eq:chaud11}
dX^1_t=F^1(X_t) \, dt + \sqrt{2} \, d\overline{B}_{\psi(t)} \quad \text{with the random clock } \ \psi(t):=\int_0^t|\Sigma^{1}|^2(X_s) \, ds,
\end{align}
where $\overline{B}$ is a one-dimensional Brownian motion given, in function of the Brownian motion $B$ of~\eqref{CaseC-eq:Pro}, by $\dps \overline{B}_u:=\int_0^{\psi^{-1}(u)}|\Sigma^{1}|^2(X_s) \, dB_s$. Using the same Brownian motion $\overline{B}$, we define the effective dynamics as 
\begin{align}\label{eq:chaud}
dZ_t= b(Z_t) \, dt +\sqrt{2} \, d\overline{B}_{\varphi(t)} \quad \text{with the random clock } \ \varphi(t):=\int_0^t\sigma^2(Z_s) \, ds,
\end{align}
where $b, \sigma:\R\rightarrow\R$ are given by~\eqref{CaseC-E1:Coeff}, that is
\begin{align*}
b(x^1):=\int_{\R^{d-1}} F^1(x^1,x^2_d) \, \overline{\mu}_{x^1}(dx^2_d) \ \ \text{ and } \ \ \sigma^2(x^1):=\int_{\R^{d-1}} |\Sigma^{1}|^2(x^1,x^2_d) \, \overline{\mu}_{x^1}(dx^2_d).
\end{align*}
We refer to~\cite[Section 6.1]{EthierKurtz09} for the well-posedness of the effective dynamics~\eqref{eq:chaud}. Estimating the difference between~\eqref{eq:chaud11} and~\eqref{eq:chaud}, we obtain
\begin{multline}\label{E2-We-PreGron}
\E\pra{\sup\limits_{t\in [0,T]}\abs{X^1_t-Z_t}^2}
\leq
4\E\pra{\sup\limits_{t\in [0,T]}\abs{\int_0^t f(X_s) \, ds}^2}+4\E\pra{\sup\limits_{t\in [0,T]}\abs{\int_0^t \bra{ b(X^1_s)- b(Z_s)}ds}^2}
\\
+8\E\pra{\sup\limits_{t\in [0,T]}\abs{\overline{B}_\psi(t) -\overline{B}_{\theta(t)}}^2} +8\E\pra{\sup\limits_{t\in [0,T]}\abs{\overline{B}_{\theta(t)}-\overline{B}_{\varphi(t)}}^2},
\end{multline} 
where $f:\R\rightarrow\R$ is given by $f(x)=F^1(x) - b(x^1)$ and $\dps \theta(t):=\int_0^t \sigma^2(X^1_s) \, ds$. We now estimate the third and fourth terms above.

Using the H{\"o}lder continuity of Brownian paths (see~\cite[Theorem 1.1]{SchillingPartzsch14} for details), we know that
\begin{align*}
\forall \gamma \in (0,1), \ \ \exists C_\gamma, \ \ \forall (s,t) \in [0,T]^2, \quad \Big| \overline{B}_t-\overline{B}_s \Big|^2 \leq C_\gamma \, |t-s|^\gamma,
\end{align*}
where the H{\"o}lder constant $C_\gamma$ has bounded moments of all orders. Applying this result to the third term of the right hand side of~\eqref{E2-We-PreGron} and using the H{\"o}lder inequality, we arrive at
\begin{multline} \label{eq:chaud14}
\E\pra{\sup\limits_{t\in [0,T]}\abs{\overline{B}_{\psi(t)}-\overline{B}_{\theta(t)}}^2}
\leq
\E\pra{C_\gamma\sup\limits_{t\in [0,T]}\abs{\psi(t)-\theta(t)}^\gamma}
\leq
\widetilde{C}_\gamma \bra{\E\pra{\sup\limits_{t\in [0,T]}\abs{\psi(t)-\theta(t)}^2}}^{\gamma/2}
\\
=
\widetilde{C}_\gamma \bra{\E\pra{\sup\limits_{t\in[0,T]}\abs{\int_0^t\pra{|\Sigma^{1}|^2(X_s)-\sigma^2(X^1_s)}ds}^2}}^{\gamma/2}
=
\widetilde{C}_\gamma \bra{ \E\pra{\sup\limits_{t\in [0,T]}\abs{\int_0^t h(X_s) \, ds}^2 } }^{\gamma/2},
\end{multline}
where $h:\R^d\rightarrow\R$ is defined by $h(x):=|\Sigma^{1}|^2(x)-\sigma^2(x^1)$ and $\dps \widetilde{C}_{\gamma} := \left[ \E(C_\gamma^{2/(2-\gamma)}) \right]^{(2-\gamma)/2}$.

We next turn to the fourth term of the right hand side of~\eqref{E2-We-PreGron}. Using the same arguments as in~\eqref{eq:chaud14}, we write
\begin{equation} \label{eq:chaud16}
\E\pra{\sup\limits_{t\in [0,T]}\abs{\overline{B}_{\theta(t)}-\overline{B}_{\varphi(t)}}^2}
\leq
\widetilde{C}_\gamma \bra{\E\pra{\sup\limits_{t\in[0,T]}\abs{\int_0^t\pra{\sigma^2(X^1_s) - \sigma^2(Z_s) }ds}^2}}^{\gamma/2}.
\end{equation}
We now assume that
\begin{equation} \label{eq:chaud19}
\sigma^2(z) = (\sigma_{\rm var}(z))^2 + (\sigma_{\rm lip}(z))^2,
\end{equation}
where $\sigma_{\rm lip}^2$ is assumed to be a Lipschitz function with a small Lipschitz constant $\nu_{\rm lip}$ and where the variance of $\sigma_{\rm var}^2$, defined by
\begin{equation} \label{eq:chaud20}
\nu_{var} := \int_\R \Big[ (\sigma_{\rm var}(z))^2 - \overline{\sigma_{\rm var}}^2 \Big]^2 \ \widehat\mu(dz), \qquad \overline{\sigma_{\rm var}}^2 := \int_\R \big( \sigma_{\rm var}(z) \big)^2 \, \widehat\mu(dz)
\end{equation}
with $\widehat\mu=\xi_\#\mu$, is assumed to be small. As an example, the function $\sigma^2(z) = (1 + \vep \sin(z/\vep)) + \vep \, |z|$ is a function which is not Lipschitz with a small constant (because of the first term), and the variance of which may not be small (because of the second term). However, it satisfies~\eqref{eq:chaud19}--\eqref{eq:chaud20} with small constants $\nu_{\rm lip}$ and $\nu_{var}$. We deduce from~\eqref{eq:chaud16} and~\eqref{eq:chaud19} that
\begin{multline} \label{eq:chaud17}
\E\pra{\sup\limits_{t\in [0,T]}\abs{\overline{B}_{\theta(t)}-\overline{B}_{\varphi(t)}}^2}
\\ \leq
\overline{C}_\gamma \left\{ \left( \E\pra{\sup\limits_{t\in[0,T]}\abs{\int_0^t\pra{\sigma^2_{\rm var}(X^1_s) - \sigma^2_{\rm var}(Z_s) }ds}}^2 \right)^{\gamma/2} + \left( \E\pra{\sup\limits_{t\in[0,T]}\abs{\int_0^t\pra{\sigma^2_{\rm lip}(X^1_s) - \sigma^2_{\rm lip}(Z_s) }ds}^2} \right)^{\gamma/2} \right\}
\end{multline}
where $\overline{C}_\gamma$ only depends on $\gamma$. To bound the first term of~\eqref{eq:chaud17}, we proceed as follows. Let $\widetilde{\sigma}(z) = \sigma^2_{\rm var}(z) - \overline{\sigma_{\rm var}}^2$. We have
\begin{align*}
\E\pra{\sup\limits_{t\in[0,T]}\abs{\int_0^t\pra{\sigma^2_{\rm var}(X^1_s) - \sigma^2_{\rm var}(Z_s) }ds}^2}
& =
\E\pra{\sup\limits_{t\in[0,T]}\abs{\int_0^t\pra{\widetilde{\sigma}(X^1_s) - \widetilde{\sigma}(Z_s) }ds}^2}
\\
& \leq
T \, \E\pra{ \int_0^T \left| \widetilde{\sigma}(X^1_s) - \widetilde{\sigma}(Z_s) \right|^2 ds}
\\
& \leq
2T \, \bra{\E\pra{ \int_0^T \left[ \widetilde{\sigma}(X^1_s) \right]^2 + \left[ \widetilde{\sigma}(Z_s) \right]^2 ds}}.
\end{align*}
Using the fact that the system starts at equilibrium, we deduce that
\begin{equation}\label{eq:chaud15}
\E\pra{\sup\limits_{t\in[0,T]}\abs{\int_0^t\pra{\sigma^2_{\rm var}(X^1_s) - \sigma^2_{\rm var}(Z_s) }ds}^2}
\leq
(2T)^2 \int_\R \left[ \widetilde{\sigma}(z) \right]^2 \widehat\mu(dz)
=
(2T)^2 \, \nu_{var}
\end{equation}
where $\widehat\mu=\xi_\#\mu$. We bound the second term of~\eqref{eq:chaud17} by simply using that $\sigma_{\rm lip}^2$ is a Lipschitz function. We thus deduce from~\eqref{eq:chaud17} and~\eqref{eq:chaud15} that
\begin{multline} \label{eq:chaud18}
\E\pra{\sup\limits_{t\in [0,T]}\abs{\overline{B}_{\theta(t)}-\overline{B}_{\varphi(t)}}^2}
\leq
\overline{C}_{\gamma,T} \left\{ \nu_{\rm var}^{\gamma/2} + \left( \nu_{\rm lip}^2 \, \E\pra{ \int_0^T \pra{X^1_s - Z_s}^2 ds} \right)^{\gamma/2} \right\}
\\
\leq
\overline{C}_{\gamma,T} \left\{ \nu_{\rm var}^{\gamma/2} + \nu_{\rm lip}^{\beta_\gamma} + \E\pra{ \int_0^T \pra{X^1_s - Z_s}^2 ds} \right\}
\end{multline}
with $\beta_\gamma = \gamma/(1-\gamma/2)$. Collecting~\eqref{E2-We-PreGron}, \eqref{eq:chaud14} and~\eqref{eq:chaud18}, and under the assumption that $b$ is Lipschitz, we deduce that
\begin{multline}\label{rem:PreGron-DubSch}
\E\pra{\sup\limits_{t\in [0,T]}\abs{X^1_t-Z_t}^2}
\leq
C \Biggl[ \E \left( \int_0^T \abs{X^1_s-Z_s}^2ds \right) + \E \left( \sup\limits_{t\in [0,T]}\abs{\int_0^t f(X_s) \, ds}^2 \right) \\ + \left(\E\pra{\sup\limits_{t\in [0,T]}\abs{\int_0^t h(X_s) \, ds}^2 }\right)^{\gamma/2} + \nu_{\rm var}^{\gamma/2} + \nu_{\rm lip}^{\beta_\gamma} \Biggr],
\end{multline}
where the constant $C>0$ only depends on $T$, $\gamma$ and the Lipschitz constant of $b$. In comparison with~\eqref{CaseC-GronGen}, the integrals with Brownian motions have been replaced by the last three terms in~\eqref{rem:PreGron-DubSch}. We can now apply the result of Lemma~\ref{CaseC-lem:LZ} to the third term of~\eqref{rem:PreGron-DubSch}. Repeating calculations as in Theorem~\ref{CaseC-thm:Strong}, and using the Gronwall argument, we arrive at
\begin{equation}\label{eq:chaud13}
\E\pra{\sup\limits_{t\in [0,T]}\abs{X^1_t-Z_t}^2} \leq C \left[ \frac{\kappa^2}{\alpha_{\PI}^2}+\bra{\frac{\lambda^2}{\alpha^2_{\PI}}}^{\gamma/2} + \nu_{\rm var}^{\gamma/2} + \nu_{\rm lip}^{\beta_\gamma} \right],
\end{equation}
for some $C>0$ (depending on $\gamma$) and any $0<\gamma<1$. Comparing this to Theorem~\ref{CaseC-thm:Strong}, we note that the error estimate does not improve in terms of powers of $\alpha_{\PI}$, although the new projected and effective dynamics (based on Dubins-Schwarz theorem) allow us to use the Lyons-Zhang argument to estimate the diffusion term. The last two terms of~\eqref{eq:chaud13} even do not depend on $\alpha_{\PI}$. Nevertheless, in some cases, the estimate~\eqref{eq:chaud13} can be sharper than that given in Theorem~\ref{CaseC-thm:Strong}, as shown in the following Remark~\ref{rem:cas_eps}.
\end{rem}

\begin{rem}\label{rem:cas_eps}
In this remark, we consider an example of SDE with an explicit parameter $\vep$ encoding the time scale separation, and show how our results above (and in particular Remark~\ref{rem:random_clock} and bound~\eqref{eq:chaud13}) can be used to obtain a quantitative error bound (see~\cite[Section 7]{LegollLelievreOlla17} for a similar analysis in the reversible setting). In $\R^2$, consider a probability measure $\mu(dxdy) = \exp(-V(x,y)) \, dxdy$, a vector field $c(x,y) = (c^1(x,y),0)^T$ such that $\div(c \, \mu)=0$ (for instance, one can choose $c^1 = \exp(V)$), and two functions $\Sigma^1, \Sigma^2: \R^2 \rightarrow \R$. We next define the diffusion matrix $\Sigma_\vep = \text{diag}(\Sigma^1, \vep^{-1/2} \, \Sigma^2)$ and $A_\vep = \Sigma_\vep\Sigma_\vep^T$. Following~\eqref{B-eq:StatMeas-Cond}, we introduce the vector field
$$
F_\vep = -A_\vep \nabla V +\div A_\vep + c = (F^1,\vep^{-1} F^2)^T
$$
with $F^1 = - (\Sigma^1)^2 \partial_x V + c^1 + \partial_x [(\Sigma^1)^2]$ and $F^2 = -(\Sigma^2)^2 \partial_y V  + \partial_y [(\Sigma^2)^2]$, so that the invariant measure of~\eqref{CaseC-eq:SDE} is the measure $\mu$ whatever the value of $\vep$. It is easy to observe that the assumptions~\ref{CaseC-ass:Poinc} and~\ref{CaseC-ass:KaLa} are satisfied with constants depending on $\vep$, and which satisfy $\alpha_{\PI}^\vep = \alpha_{\PI}^0 / \vep$, $\kappa^\vep = \kappa^0 / \sqrt{\vep}$ and $\lambda^\vep = \lambda^0 / \sqrt{\vep}$, with $\alpha_{\PI}^0$, $\kappa^0$ and $\lambda^0$ independent of $\vep$. In addition, we observe that the effective coefficients $b$ and $\sigma$ are independent of $\vep$. We can thus assume that $b$ is a Lipschitz function with a constant $L_b$ independent of $\vep$. We thus deduce from~\eqref{eq:chaud13} that, for any $0<\gamma<1$, there exists some $C>0$ such that, for any $\vep$,
$$
\E\pra{\sup\limits_{t\in [0,T]}\abs{X^{\vep,1}_t-Z_t}^2} \leq C \left( \vep + \vep^{\gamma/2} + \nu_{\rm var}^{\gamma/2} + \nu_{\rm lip}^{\beta_\gamma} \right)
$$
with $\nu_{\rm var}$ and $\nu_{\rm lip}$ defined by~\eqref{eq:chaud19}--\eqref{eq:chaud20}. 
Supppose now that the effective drift $\sigma$ is constant, so that the last two terms above vanish. Then, as expected, the projected dynamics converges, when $\vep \to 0$, to the effective dynamics, and our approach yields a quantitative error bound. The connections between averaging techniques and effective dynamics in the context of non-reversible SDEs will be discussed in more details in the forthcoming article~\cite{HNS}.
\end{rem}

In the following remark, we discuss generalisations to vector-valued affine coarse-graining maps.

\begin{rem}\label{rem:GenCG-B}
Consider the SDE~\eqref{CaseC-eq:SDE}, that is
\begin{align*}
dX_t=F(X_t) \, dt + \sqrt{2} \, \Sigma(X_t) \, dW_t, \quad X_{t=0}=X_0,
\end{align*}
and the case of a {\em vectorial} affine coarse-graining map $\xi:\R^d\rightarrow\R^k$ with $\xi(x)=\mathbb Tx+\tau$, where $\tau\in\R^k$ and $\mathbb T\in\R^{k\times d}$ is of full rank (see Remark~\ref{rem:GenCG-A} for the case $\Sigma = \Id_d$). With the notation $\widehat X_t:=\xi(X_t)$, the projected and the effective dynamics are
\begin{align} \label{eq:chaud12}
d\widehat X_t=(\nabla\xi \, F)(X_t) \, dt + \sqrt{2} \, \widehat\Sigma(X_t) \, dB_t, \qquad dZ_t = b(Z_t) \, dt + \sqrt{2} \, \sigma(Z_t) \, dB_t,
\end{align}
where the diffusion matrix $\widehat\Sigma^2:=\nabla\xi\Sigma\Sigma^T\nabla\xi^T$ is a $\R^{k\times k}$ matrix and $B_t$ is a $k$-dimensional Brownian motion given by $\dps dB_t= (\widehat\Sigma)^{-1} \, \nabla\xi\Sigma \, dW_t$.

The effective coefficients $b:\R^k\rightarrow\R^k$ and $\sigma:\R^k\rightarrow\R^{k\times k}$ in~\eqref{eq:chaud12} are given by
\begin{align*}
\forall z\in\R^k, \quad b(z):=\int_{\xi^{-1}(z)} (\nabla\xi \, F)(y) \, \overline{\mu}_{z}(dy), \qquad \sigma^2(z):=\int_{\xi^{-1}(z)} \widehat \Sigma^2(y) \, \overline{\mu}_{z}(dy).
\end{align*}
As before $\mu(\cdot \, | \, \xi(x)=z)=:\overline{\mu}_z(\cdot)\in\mathcal P(\xi^{-1}(z))$ is the family of conditional invariant measures. We make the following assumptions (similar to those made elsewhere in this article):
\begin{enumerate}
\item The system starts at equilibrium, i.e. $X_0\sim\mu$ and $Z_0=X^1_0$.
\item The family of measures $\overline{\mu}_{z} \in \mathcal P(\xi^{-1}(z))$ satisfies a Poincar{\'e} inequality uniformly in $z\in\R^k$ with constant $\alpha_{\PI}$, in the sense that, for any $z\in\R^k$ and $h:\xi^{-1}(z)\rightarrow\R$ for which the right hand side below is finite, we have
$$
\int_{\xi^{-1}(z)} \Bigr(h-\int_{\xi^{-1}(z)}h \, \overline{\mu}_z(dy) \Bigl)^2 \, \overline{\mu}_z(dy) \leq \frac{1}{\alpha_{\PI}} \int_{\xi^{-1}(z)} (A \Pi \nabla h) \cdot (\Pi \nabla h) \, \overline{\mu}_z(dy),
$$
where $A$ is defined by~\eqref{eq:defA} and $\Pi$ is defined (see~\cite[Equation~(23)]{LelievreZhang18} for details) by
\begin{equation}
\Pi := \Id_d - \sum_{1 \leq i,j \leq k} (\widehat\Sigma^{-1})_{ij} \nabla\xi_i \otimes (A\nabla\xi_j).
\end{equation}
\item The effective coefficients $b$ and $\sigma$ are Lipschitz.
\item The following quantities are finite:
\begin{align*}
\kappa^2:= \sum_{i=1}^k \int_{\R^d} \left[ A \Pi \nabla (\nabla\xi \, F)_i \right] \cdot \left[ \Pi \nabla (\nabla\xi \, F)_i \right] \, \mu <\infty,
\qquad
\lambda^2 := \sum_{i,j=1}^k \int_{\R^{d}} \left[ A \Pi \nabla \widehat\Sigma_{ij} \right] \cdot \left[ \Pi \nabla \widehat\Sigma_{ij} \right] \, \mu  < \infty.
\end{align*}
\end{enumerate}
Using in particular~\cite[Lemma~5]{LelievreZhang18}, we can extend the proof of Theorem~\ref{CaseC-thm:Strong} to this case and obtain that
\begin{align*}
\E\pra{\sup\limits_{t\in [0,T]}\abs{X^1_t-Z_t}^2}\leq C_1 \, e^{C_2T} \, \bra{\frac{\kappa^2}{\alpha^2_{\PI}} + \frac{\lambda^2}{\alpha_{\PI}}},
\end{align*}   
where $C_1$ and $C_2$ are independent of $T$, $\kappa$, $\lambda$ and $\alpha_{\PI}$.
\end{rem}

\textbf{Acknowledgements} The work of the authors is supported by the European Research Council under the European Union's Seventh Framework Programme (FP/2007-2013)/ERC Grant Agreement number 614492. The authors would like to thank Mark Peletier for useful discussions on a preliminary version of this work, and Bernard Derrida, Carsten Hartmann, Lara Neureither, Julien Reygner and Wei Zhang for stimulating interactions. Part of this work was performed while TL and US were visiting the International Centre for Theoretical Sciences (ICTS Bengaluru). TL and US would like to thank ICTS for its hospitality. 
 
\bibliographystyle{alphainitials}
\bibliography{Bib_PE}
%\bibliographystyle{../alphainitials}
%\bibliography{../Bib_PE}

\vspace{0.5cm}

(F. Legoll) {\'E}cole des Ponts ParisTech and Inria, 6-8 Avenue Blaise Pascal, Cit{\'e} Descartes, 77455 Marne-la-Vall{\'e}e, France \\
E-mail address: \href{mailto:frederic.legoll@enpc.fr}{frederic.legoll@enpc.fr}

(T. Leli{\`e}vre) {\'E}cole des Ponts ParisTech and Inria, 6-8 Avenue Blaise Pascal, Cit{\'e} Descartes, 77455 Marne-la-Vall{\'e}e, France \\
E-mail address: \href{mailto:tony.lelievre@enpc.fr}{tony.lelievre@enpc.fr}

(U. Sharma) {\'E}cole des Ponts ParisTech, 6-8 Avenue Blaise Pascal, Cit{\'e} Descartes, 77455 Marne-la-Vall{\'e}e, France \\
E-mail address: \href{mailto:upanshu.sharma@enpc.fr}{upanshu.sharma@enpc.fr}

\end{document}